\documentclass[11pt]{article}   

\usepackage{graphicx}

\usepackage{latexsym,amsfonts,amsmath,theorem,amssymb}
\usepackage{stmaryrd,array,tabularx,bbm}
\usepackage{pstricks,graphicx}

\usepackage{a4wide}

\usepackage{color}
\usepackage{caption}
\usepackage{theorem}
\usepackage{hyperref}
\usepackage{cleveref}
\usepackage{framed}
\usepackage{blindtext}
\usepackage{float}
\usepackage{todonotes}
\usepackage{mathtools}

\usepackage[caption=false]{subfig}
\usepackage{algorithm}
\usepackage{algorithmicx}
\usepackage[noend]{algpseudocode}
\usepackage{tikz}

\algrenewcommand\algorithmicrequire{\textbf{Input:}}
\algrenewcommand\algorithmicensure{\textbf{Output:}}
\newtheorem{theorem}{Theorem}[section]
\newtheorem{lemma}[theorem]{Lemma}
\newtheorem{proposition}[theorem]{Proposition}

\newenvironment{proof}{\begin{trivlist}
    \item[\hskip\labelsep{\bf Proof.}]}{$\hfill\Box$\end{trivlist}}

{\theoremstyle{plain} \theorembodyfont{\rmfamily}
\newtheorem{example}[theorem]{Example}
}

\numberwithin{equation}{section}
\numberwithin{figure}{section}
\numberwithin{table}{section}

      \newcommand{\tgn}[1]{{(#1)_{t\geq 0}}}

\renewcommand{\hat}{\widehat}
\newcommand{\vertiii}[1]{{\left\vert\kern-0.25ex\left\vert\kern-0.25ex\left\vert #1 
\right\vert\kern-0.25ex\right\vert\kern-0.25ex\right\vert}}

\title{The random timestep Euler method and its continuous dynamics}

\author{Jonas Latz}
\date{\footnotesize Department of Mathematics, University of Manchester, UK\\
\url{jonas.latz@manchester.ac.uk}}

\begin{document}
\maketitle
\begin{abstract}
ODE solvers with randomly sampled timestep sizes appear in the context of chaotic dynamical systems, differential equations with low regularity, and, implicitly, in stochastic optimisation. In this work, we propose and study the stochastic Euler dynamics -- a continuous-time Markov process that is equivalent to a linear spline interpolation of a random timestep (forward) Euler method. We understand the stochastic Euler dynamics as a path-valued ansatz for the ODE solution that shall be approximated.
We first obtain qualitative insights by studying deterministic Euler dynamics which we derive through a first order approximation to the infinitesimal generator of the stochastic Euler dynamics. Then we show convergence of the stochastic Euler dynamics to the ODE solution  by studying the associated infinitesimal generators and by a novel local truncation error analysis. Next we prove stability by an immediate analysis of the random timestep Euler method and by deriving Foster--Lyapunov criteria for the stochastic Euler dynamics; the latter also yield bounds on the speed of convergence to stationarity. The paper ends with a discussion of second-order stochastic Euler dynamics and a series of numerical experiments that appear to verify our analytical results.
\end{abstract}
\noindent\textbf{Keywords:} numerical methods for ordinary differential equations, piecewise deterministic Markov processes, stochastic algorithms

\noindent\textbf{MSC2020:} 	65L05,  	
	60J25,   
	65C99  	
\section{Introduction}
Ordinary and partial differential equations are used to  model systems and processes in, e.g., the physical and social sciences. There, they sometimes arise as limits of certain stochastic processes or particle systems. Examples are the heat equation that describes the behaviour of an underlying Brownian motion, see, e.g., \cite{Pavl}, chemical reaction networks that are often represented by systems of ODEs arising at the large-number-of-reactions-limit of underlying Markov pure jump processes, see, e.g., \cite{Darling,Kurtz}, or non-linear Fokker--Planck equations that describe the distribution of pedestrians that move in space according to certain McKean--Vlasov stochastic processes, see \cite{Gomes}.
When models of this form need to be simulated, both the stochastic processes and their differential equation limit can be taken into account: whilst the basic heat equation in low dimensions is usually accurately simulated with the  finite element method \cite{Iserles}, Yang et al. \cite{nonlocal} propose a sampling-based approach for closely-related, semilinear non-local diffusion equations.

In this work, we propose and  study a class of continuous-time Markov processes (CTMPs) that can be used as an approximation scheme for  a large class of ordinary differential equations. We obtain these so-called \emph{stochastic Euler dynamics}  by employing the forward Euler method  with exponentially distributed random timestep sizes and by then linearly interpolating in between the timesteps. This procedure defines a continuous-time Markov process and can be generalised to certain other random timestep ODE solvers. Convergence occurs when the random timesteps converge to zero, which is equivalent to the large-number-of-reactions-limit in chemical reaction networks. We understand our class of CTMPs as a randomised ansatz to the ODE solution, similar to a Gaussian process ansatz in the context of ODEs (e.g., \cite{Kersting}),  PDEs (e.g., \cite{CHEN2021110668,MENG2023112340}), or general function approximation (e.g., \cite{Lin24,stuart2018posterior}) or the random sampling in randomised linear algebra (e.g., \cite{HMT,Nakatsukasa}). In our ansatz, we assume that the ODE can be approximated by a piecewise linear path where the intervals on which the path is linear are random -- the interval lengths are exponentially distributed. 

Random timestep sizes have been motivated by Abdulle and Garegnani \cite{abdulle}  in the context of chaotic dynamics, where the choice of timestep size may have a large influence on the simulation outcome. The intuitive idea is that repeated simulations with random timesteps and the consideration of the distribution of their outcomes may actually be more indicative of the dynamical system's behaviour than a single unstable trajectory. Outside of chaotic dynamics, random timesteps are used for differential equations with irregular right-hand sides: Eisenmann et al. \cite{Eisenmann}, for instance, study them in the case of time-irregular coefficients, whereas \cite{JENTZEN2009346,STENGLE199025} consider them in the discretisation of Carath\'eodory differential equations. 
Random timestep methods have also been considered within stochastic optimisation methods in machine learning, see, e.g., \cite{jin2022,Latz2021,Mignacco,yoshida2024}. These methods fall into the context of our work, as many machine learning methods are based on the gradient descent method, which is a forward Euler timestepping of a gradient flow \cite{Nocedal}. In machine learning, random timesteps can lead to a regularisation of the underlying learning problem. 

The goal of our work is to derive numerical properties of the random timestep Euler method through, both, the lens of continuous-time Markov processes and the lens of numerical analysis for ODE solvers: deriving novel local truncation error estimates and proving convergence by comparing infinitesimal generators, studying stability by studying mean and second moment of the generated trajectory and through continuous-time Foster--Lyapunov criteria, as well as deriving higher-order methods through polynomial splines and the construction of ODEs through approximations of infinitesimal operators. The continuous-time viewpoint is vital in the context of random timestep ODE solvers: as the evaluation times are random, it is difficult to compare different solver trajectories without interpolation. Our approach contains this interpolation naturally.

We now introduce our problem setting, the random timestep Euler method, and the stochastic Euler dynamics. Then, we list our contributions and outline this work.

\subsection{The random timestep Euler method and the stochastic Euler dynamics} 
Let $X:=\mathbb{R}^d$ for some $d \in \mathbb{N}:= \{1,2,...\}$, let $\langle \cdot, \cdot\rangle$ be the Euclidean inner product, and $\|\cdot\|$ be the associated norm. Let $f: X \rightarrow X$ be a Lipschitz continuous function. Throughout this work, we study autonomous ordinary differential equations of the form
\begin{align}
    u'(t) &= f(u(t)) \quad (t > 0), \label{eq_ODE}\\
    u(0) &= u_0. \notag
\end{align}
 The \emph{forward Euler method} approximates the trajectory $(u(t))_{t \geq 0}$ by a sequence 
\begin{equation} \label{eq_Euler_disc_det}
    \hat{u}_k =  \hat{u}_{k-1} + h_k f(\hat{u}_{k-1}) \quad (k \in \mathbb{N}), \qquad \hat{u}_0 = u_0,
\end{equation}
where $(h_k)_{k \in \mathbb{N}} \in (0,\infty)^{\mathbb{N}}$ is a sequence of stepsizes and $\hat{u}_K$ approximates $u\left(t_K\right)$, with $t_K := \sum_{k=1}^Kh_k$ and  $K \in \mathbb{N}$. Between two timesteps, $u(\cdot)$ can be approximated by linear interpolation: $$u(t) \approx \hat{u}(t) :=  \frac{t_{K}-t}{t_{K} - t_{K-1}}\hat{u}_{K-1} + \frac{t-t_{K-1}}{t_{K} - t_{K-1}}\hat{u}_{K} \qquad (t \in [t_{K-1},t_{K}], K \in \mathbb{N}).$$
The path $(\hat{u}(t))_{t \geq 0}$ is a polygonal chain (a piecewise linear process): it is linear in-between the discretisation points $(t_k)_{k = 0}^\infty$ and continuous at the discretisation points.

In the following, we choose the stepsizes  to be independent and identically distributed (i.i.d.) exponential random variables $H_1, H_2,\ldots \sim \mathrm{Exp}(1/h)$, with  $h = \mathbb{E}[H_1] > 0$ being the \emph{stepsize parameter}. Then, we define the \emph{random timestep (forward) Euler method} through
\begin{equation} \label{eq_Euler_RT}
    \hat{V}_k =  \hat{V}_{k-1} + H_k f(\hat{V}_{k-1}) \quad (k \in \mathbb{N}), \qquad \hat{V}_0 = u_0.
\end{equation}
These and all other random variables throughout this work are defined on an underlying probability space $(\Omega, \mathcal{F}, \mathbb{P})$. Whilst \cite{abdulle} allow for a more general distribution of the $(H_k)_{k=1}^\infty$, we require exponentially distributed stepsizes:
When choosing exponentially distributed stepsizes we can formulate the linearly interpolated version (analogous to $\hat{u}(\cdot)$) of the random timestep Euler method $(\hat{V}_k)_{k=0}^\infty$ as a (homogeneous-in-time) continuous-time Markov process, which we refer to as \emph{stochastic Euler dynamics}. This is easy to see when we formulate the stochastic Euler dynamics as an ODE with jumps after exponential waiting times. We denote the random jump times by $T_1, T_2, \ldots$, where $T_k -T_{k-1} = H_k$ $(k \in \mathbb{N})$ and $T_0 = 0$. Then, we define the stochastic Euler dynamics as $(V(t), \overline{V}(t))_{t \geq 0}$ by
\begin{align} \label{eq_SED}
    {V}'(t) &= f(\overline{V}(t))  &(t \in (T_{k-1}, T_k], k \in \mathbb{N}) \notag \\
    \overline{V}'(t) &= 0  &(t \in (T_{k-1}, T_k], k \in \mathbb{N}) \notag \\
    V(T_{k-1}) &= V(T_{k-1}-) &(k \in \mathbb{N}) \\ \notag
     \overline{V}(T_{k-1}) &=  {V}(T_{k-1}-) &(k \in \mathbb{N}) \\
    V(0) &=  \overline{V}(0) = u_0 \notag
\end{align}
where we denote $g(t-) = \lim_{t_0 \uparrow t}g(t_0)$ for some function $g$ on $\mathbb{R}$. This process is well-defined as it consists of subsequent well-defined ODEs, see, e.g. \cite{Latz2021}. The associated jump chain $(V(T_k))_{k =0}^\infty = (\overline{V}(T_k))_{k =0}^\infty$ is, of course, identical to the trajectory  $(\hat{V}_k)_{k=0}^\infty$ of the random timestep Euler method \eqref{eq_Euler_RT}.
We illustrate the stochastic Euler dynamics  in Figure~\ref{fig:SED_examples} and note that its two components can be interpreted as follows:
\begin{itemize}
    \item  $(\overline{V}(t))_{t \geq 0}$ is a piecewise constant process, that jumps to the value of $(V(t))_{t\geq 0}$ at the jump time and, thus, memorises this value until the next jump time. We refer to $(\overline{V}(t))_{t \geq 0}$ as \emph{companion process}.
    \item $(V(t))_{t \geq 0}$ is the piecewise linear path that corresponds to the linearly interpolated Euler trajectory $(\hat{u}(t))_{t \geq 0}$ where the $\hat{u}_K = \hat{V}_K$ are obtained from the random timestep Euler method, for $K \in \mathbb{N}_0 := \mathbb{N} \cup \{0\}$. It is linear on the jump time intervals $(T_{k-1}, T_k]$  as $t \mapsto f(\overline{V}(t))$ is constant on those intervals. We sometimes refer also to $({V}(t))_{t \geq 0}$ as \emph{stochastic Euler dynamics}.
\end{itemize}

\begin{figure}
    \centering
    \includegraphics[scale=0.85]{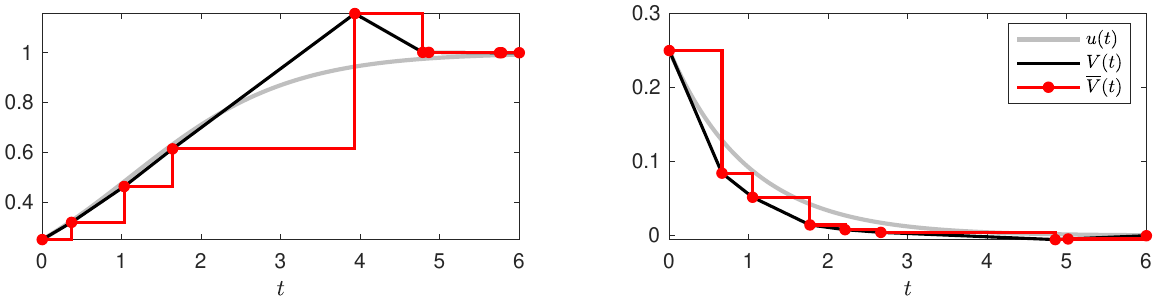}
    \caption{Realisations of the stochastic Euler dynamics considering the ODEs $u'=(1-u)u$ (left) and $u'=-u$ (right) with initial values $u_0 = 0.25$ and stepsize parameter $h= 0.8$.}
    \label{fig:SED_examples}
\end{figure}

The stochastic Euler dynamics $(V(t),  \overline{V}(t))_{t \geq 0}$ is an ODE with random jumps and, thus, falls into the category of piecewise deterministic Markov processes (PDMPs) as introduced by Davis \cite{Davis84}. PDMPs are non-diffusive stochastic processes that have appeared in, e.g., biology \cite{rudnicki2017piecewise}, computational statistics \cite{Zigzag,Boomerang}, and optimisation theory \cite{jin2022}. Especially the Zigzag process \cite{Zigzag} is conceptually similar to the stochastic Euler dynamics, as both processes are polygonal chains. PDMPs are Feller processes and can, thus, be represented by infinitesimal generators. We denote the infinitesimal generator of the stochastic Euler dynamics by $\mathcal{A}_h: C^1(X \times X) \rightarrow C^0(X \times X)$. It is defined by
\begin{align*}
    \mathcal{A}_h \varphi(v,\overline{v}) :&= \lim_{t \downarrow 0} \frac{1}{t}\left(\mathbb{E}[\varphi(V(t), \overline{V}(t))|V(0) = v, \overline{V}(0) = \overline{v}] - \varphi(v, \overline{v})\right) \\
    &= \langle \nabla_v \varphi(v,\overline{v}), f(\overline{v})\rangle + (\varphi (v,v) - \varphi(v,\overline{v}))/h \qquad (v,\overline{v} \in X),
\end{align*}
where $\varphi \in C^1(X\times X)$ is a test function. As usual, $C^k(A; B)$ denotes the Banach space of $k$-times continuously differentiable functions $g:A \rightarrow B$ that is equipped with the usual supremum norm $\| g \|_{\infty} = \sum_{i=0}^k\sup_{x \in A}|g^{(i)}(x)|$ and we write $C^k(A) := C^k(A; \mathbb{R}).$
In a similar way, we can also obtain the    infinitesimal generator of the ODE \eqref{eq_ODE}. There, we have 
$$
\mathcal{A} \psi(u) :=  \langle \nabla_u \psi(u), f(u) \rangle  \qquad (u \in X),
$$
where $\psi \in C^1(X)$. In a way, we see $\mathcal{A}_h$ as an approximation of $\mathcal{A}$. Infinitesimal generators play an important role in our analysis, especially when we study convergence of the paths of the stochastic Euler dynamics to the ODE solution and when we consider Foster--Lyapunov criteria. For other uses of infinitesimal generators, we refer to, e.g.,  \cite{lasota2013chaos,Pavl}.

\subsection{Contributions and outline}
The purpose of this work is an analysis of the random timestep Euler method and the stochastic Euler dynamics. We begin by proposing and analysing the \emph{deterministic Euler dynamics}, a system of ordinary differential equations  arising from an approximation of the infinitesimal operator $\mathcal{A}_h$. This continuous-time dynamical system is closely related to the forward Euler method:
\begin{itemize}
    \item we prove stability of the deterministic Euler dynamics as well as convergence to \eqref{eq_ODE} in the small stepsize parameter limit with rate $O(h; h\downarrow 0)$ in the case of a linear ODE
\end{itemize}
We then study properties of the stochastic Euler dynamics. Firstly, we consider the \emph{convergence} of the stochastic Euler dynamics to the ODE \eqref{eq_ODE} using two approaches:
\begin{itemize}
    \item we prove the weak convergence of the path $(V(t))_{t \geq 0}$ to $(u(t))_{t \geq 0}$ with respect to a weighted uniform metric where $f$ in \eqref{eq_ODE} only needs to be Lipschitz continuous
    \item we introduce a notion of local truncation error for the stochastic Euler dynamics and analyse it in the case of linear ODEs; we especially show that this \emph{local root mean square truncation error} behaves as $\mathbb{E}[\|V(\varepsilon) - u(\varepsilon)\|^2]^{1/2} = O(\varepsilon^2; \varepsilon \downarrow 0)$
\end{itemize}
Secondly, we study the \emph{stability} of the stochastic Euler dynamics in the case of linear ODEs $u' = Au$ with symmetric $A$:
\begin{itemize}
    \item we prove stability of the random timestep Euler method by analysing mean and second moment of the discrete chain
    \item we prove stability of the stochastic Euler dynamics using Foster--Lyapunov criteria and give sharp bounds on its speed of convergence to stationarity
\end{itemize}
Finally, to extend our theoretical work on the random timestep Euler method and the stochastic Euler dynamics,
\begin{itemize}
    \item we introduce the second-order stochastic Euler dynamics and study its stability
    \item we aim to verify the results from our theoretical analysis with a series of numerical experiments and explore  situations not covered by our theoretical results
\end{itemize}

This work is organised as follows: Deterministic Euler dynamics are introduced and analysed in Section~\ref{sec_DED}. Convergence and stability of stochastic Euler dynamics are subject of Sections~\ref{Sec_Converg} and \ref{Sec_Stab}. Second-order methods are discussed in Section~\ref{Sec_higherorder} before numerical experiments are presented in Section~\ref{Sec_Numerical} and the work concludes in Section~\ref{Sec_Conclusions}.

\section{Deterministic Euler dynamics} \label{sec_DED}
We begin our study of the stochastic Euler dynamics by discussing a deterministic dynamical system that we obtain through an approximation of the infinitesimal operator $\mathcal{A}_h$. Indeed, we obtain the operator ${\mathcal{A}
}_h^{\rm d}$ through a first order Taylor approximation of $\mathcal{A}_h$. It is given by
$$
{\mathcal{A}}_h^{\rm d} \varphi(w,\overline{w}) :=\langle \nabla_w \varphi(w,\overline{w}), f(\overline{w})\rangle + \langle \nabla_{\overline{w}} \varphi(w,\overline{w}), w-\overline{w} \rangle/h \qquad (w,\overline{w} \in X)
$$
for an appropriate test function $\varphi \in C^1(X\times X; \mathbb{R})$. Now, ${\mathcal{A}}_h^{\rm d}$ is the infinitesimal generator of the following ODE that we refer to as \emph{deterministic Euler dynamics}
\begin{align*}
   w'(t) &= f(\overline{w}(t)), \\
  \overline{w}'(t) &= \frac{w(t)-\overline{w}(t)}{h}, \\
  w(0) &= \overline{w}(0) = u_0.
\end{align*}
This ODE is well-defined as $f$ is Lipschitz and $(w,\overline{w}) \mapsto (w-\overline{w})/h$ is linear.
A noticeable difference between deterministic and stochastic Euler dynamics is that the companion process $\tgn{\overline{w}(t)}$ is a smooth function, where $\tgn{\overline{V}(t)}$ is a piecewise constant jump process. $\tgn{w(t),\overline{w}(t)}$ has an interesting interpretation: The trajectory $\tgn{w(t)}$ approximates the original ODE solution not by evaluating the function $f$ directly, but by relying on evaluations of the trajectory $\tgn{\overline{w}(t)}$. The companion process appears to start an infinitesimal bit slower than $\tgn{w(t)}$ as $\overline{w}'(0)= 0$; so it always trails behind $\tgn{w(t)}$. This fact appears to be used by  $\tgn{\overline{w}(t)}$ as it is defined through its derivative at time $t>0$ being equal to the finite difference $\frac{1}{h}(w(t) -\overline{w}(t))$, where ${w}(t)$ approximates $u(t)$ and $\overline{w}(t)$ approximates $u(t-h)$. We briefly showcase this behaviour for a linear ODE in one dimension.

\begin{example} \label{Ex_Lin_1D_firstorder}
Let $X = \mathbb{R}$ and $f(u) = -au$ for some $a > 0$. Then, we can find analytical expressions for $\tgn{w(t),\overline{w}(t)}$:
\begin{align*}
    w(t) &= \frac{u_0 \exp\left(-\frac{t (\sqrt{1 - 4 a h} + 1)}{2 h}\right)}{2 \sqrt{1 - 4 a h}} \left((-2 a h + \sqrt{1 - 4 a h} + 1) \exp\left(\frac{t \sqrt{1 - 4 a h}}{h}\right) + 2 a h + \sqrt{1 - 4 a h} - 1\right),\\
     \overline{w}(t) &= \frac{u_0 \exp\left(-\frac{t (\sqrt{1 - 4 a h} + 1)}{2 h}\right)}{2 \sqrt{1 - 4 a h}} \left((\sqrt{1 - 4 a h} + 1) \exp\left(\frac{t \sqrt{1 - 4 a h}}{h}\right) + \sqrt{1 - 4 a h} - 1\right) \qquad (t \geq 0).
\end{align*}
We make a few interesting observations: Firstly, we note that both $\tgn{w(t)}$ and $\tgn{\overline{w}(t)}$ approach the true solution $u(t) = \exp(-at)u_0$ as $h \downarrow 0$. Secondly, the trajectories $\tgn{w(t),\overline{w}(t)}$ change their behaviour rather drastically if $ah> 1/4$, but the dynamics is still well-defined. We illustrate this in Figure~\ref{fig:DED_examples}. There, we see that large $h$ lead to oscillatory behaviour, which also happens when we employ the forward Euler method for stepsizes in $(1/a, 2/a)$, i.e., close to instability. Thirdly, consistent with our reasoning above, we see that $\tgn{\overline{w}(t)}$ trails behind $\tgn{w(t)}$.
\end{example}

\begin{figure}
    \centering
    \includegraphics[scale=0.85]{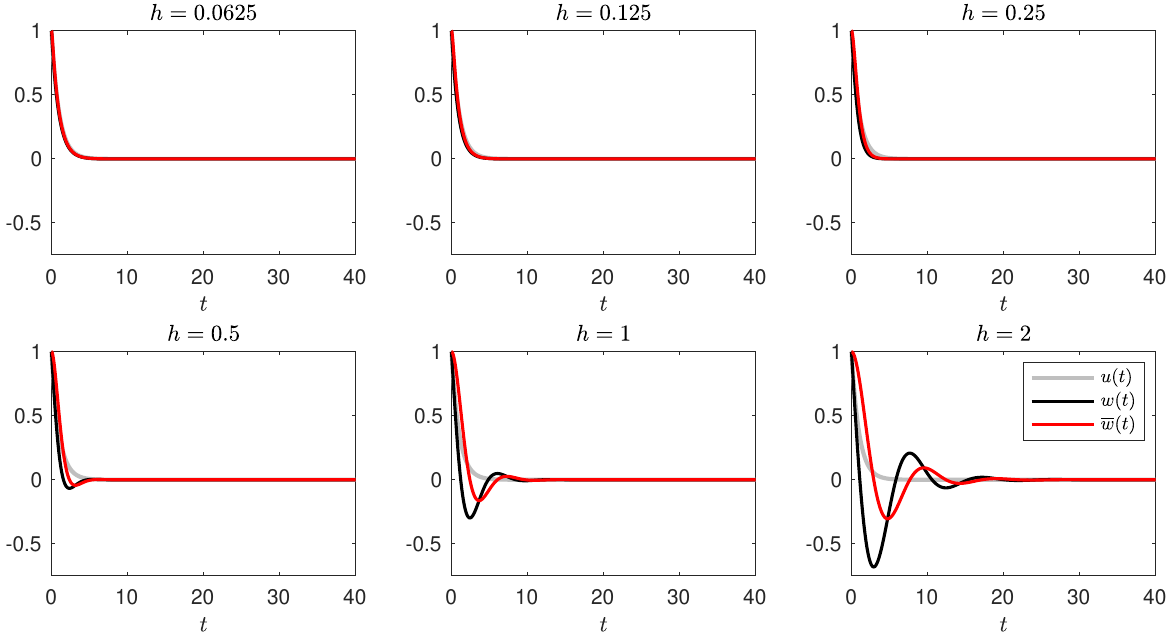}
    \caption{Plots of the deterministic Euler dynamics regarding $u'=-u$, $u_0 = 1$, and $h \in \{0.625, 0.125, \ldots, 1\}$.}
    \label{fig:DED_examples}
\end{figure}

We now try to generalise the discussion from the example. We focus on general linear ODEs, i.e., $f(u) = Au$ for a matrix $A \in \mathbb{R}^{d \times d}$. We anticipate that similar results also hold in the nonlinear case. In any case, we can state the deterministic Euler dynamics as a linear ODE of the form
\begin{equation} \label{eq_B_matrix}
    \begin{pmatrix}w'(t) \\ \overline{w}'(t)
\end{pmatrix} = B \begin{pmatrix}w(t) \\ \overline{w}(t)
\end{pmatrix}, \qquad B := \begin{pmatrix}
    0 & A \\ h^{-1}I_d & -h^{-1}I_d 
\end{pmatrix},
\end{equation}
where $I_d$ denotes the identity matrix in $\mathbb{R}^{d \times d}$.
We start by discussing the stability of this ODE under the assumption that $u' = Au$ is stable, i.e. the  eigenvalues of $A$ have negative real parts.

\begin{theorem} \label{thm:lin_stab_fo}
     Let $A \in \mathbb{R}^{d \times d}$ and let the eigenvalues $\lambda_1,\ldots,\lambda_d \in \mathbb{C}$ of $A$ have negative real parts: $\Re(\lambda_1),\ldots, \Re(\lambda_d) < 0$. Then, the eigenvalues of $B$ have negative real parts, if $h > 0$ is chosen such that
    $$
  h < -\frac{\Re(\lambda_i)}{\Im(\lambda_i)^2} \qquad (i=1,\ldots,d),
    $$
    with $-\Re(\lambda_i)/0 := \infty$.
    Indeed, if $f(u) = Au$ and the assumption above holds, there is a $c >0$ such that  the dynamical system $\tgn{w(t),\overline{w}(t)}$ satisfies $\|(w(t),\overline{w}(t))^T\| \leq \sqrt{2}\exp(-ct)\|u_0\|$ $(t \geq 0)$. 
\end{theorem}
\begin{proof}
    The eigenvalues of $B$ are the roots of the characteristic polynomial 
\begin{align*}
   p(\mu) 
   = \det\left(-\mu I_d + \frac{h^{-1}}{h^{-1}+\mu}A\right) \det(-(h^{-1}+\mu)I_d)
   = \det((\tfrac{\mu}{h}+\mu^2)I_d-\tfrac{1}{h}A).
\end{align*}
Now, given the eigenvalues of $A$, we can determine the eigenvalues of $B$ by solving
$
 h\mu^2 + \mu  - \lambda_i = 0,
$
for each $i = 1,\ldots,d$. We obtain
$
\mu = (-1 \pm \sqrt{1 + 4h\lambda_i })/({2h})
$ and by studying the principle square root of $1 + 4h \lambda_i$, we see that $\Re(\mu)<0$, if the $h \Im(\lambda_i)^2 < - \Re(\lambda_i)$.
\end{proof}
This theorem confirms our discussion from Example~\ref{Ex_Lin_1D_firstorder}, however it is rather different to usual stability theory of ODE solvers for stiff linear ODEs. If the eigenvalues of $A$ are real, i.e., $\Im(\lambda_i) = 0$ $(i=1,...,d)$, we have no restriction on $h$. Note that this is especially the case if $A$ is symmetric. Otherwise, large imaginary parts require a small stepsize $h = O(1/\Im(\lambda_i)^2)$. Isolatedly from the real parts, this observation is qualitatively similar to usual stability theory. We illustrate these facts in numerical experiments in Section~\ref{Sec_Numerical}.

We next study the deterministic Euler dynamics at the small stepsize limit $h \downarrow 0$ and especially its convergence to the ODE \eqref{eq_ODE}.

\begin{theorem} \label{thm_det}
   Let $f(u) = Au$, where $A \in \mathbb{R}^{d \times d}$, and $B$ be as given in \eqref{eq_B_matrix}.
  Then, we have for $t \geq 0$, 
  \begin{enumerate}
      \item[(i)] $\|w(t)-\overline{w}(t)\| \leq \sqrt{2}h\|\exp(tB)\|\|u_0\| $,
      \item[(ii)] $\|w'(t)-Aw(t)\|  \leq \sqrt{2}h\|\exp(tB)\|\|A\|\|u_0\|$,
      \item[(iii)] $\|w(t)-u(t)\| \leq C_th $ for some constant $C_t > 0$, and
      \item[(iv)]  $\|\overline{w}(t)-u(t)\| \leq C_t^*h $ for some constant $C_t^* > 0$.
  \end{enumerate}
\end{theorem}
\begin{proof} Let $t \geq 0$.
(i) We have $$\|w(t) - \overline{w}(t)\| = h\|\overline{w}'(t)\| \leq h \sqrt{\|w'(t)\|^2 + \|\overline{w}'(t)\|^2} = h \|\exp(tB)(u_0, u_0)^T\| \leq \sqrt{2}h\|\exp(tB)\|\|u_0\|.$$
    (ii) We have $$\|w'(t)-Aw(t)\| = \|A\overline{w}(t)-Aw(t)\| \leq \|A\|\|\overline{w}(t)-w(t)\|\stackrel{\text{(i)}}{\leq } \sqrt{2}h\|A\|\|\exp(tB)\|\|u_0\|.$$
    (iii) We have
    \begin{align*}
            \frac{\mathrm{d}}{\mathrm{d}t}\frac{1}{2}\|u(t) - w(t)\|^2 &= \langle u(t) - w(t), A(u(t) - \overline{w}(t)) \rangle \\ &= \langle u(t) - w(t), A(u(t) - w(t)) \rangle - h\langle u(t) - w(t), A \overline{w}'(t) \rangle \\ &\leq \alpha \|u(t) - w(t)\|^2 + h\|u(t) - w(t)\|\|\overline{w}'(t)\|,
    \end{align*}
    with $\alpha \in \mathbb{R}$ being the maximum eigenvalue of $\frac{1}{2}(A + A^T)$.
    Then, Gr\"onwall's inequality implies:
    $$\|u(t) - w(t)\|^2 \leq h g(t) + h\int_{0}^t g(s)\exp(\alpha(t-s))\mathrm{d}s,$$
    where $g(t) = \int_0^t \|u(s) - w(s)\|\|\overline{w}'(s)\|\mathrm{d}s$. All integrals that appear are finite as we integrate a continuous function that is bounded on the compact domain of integration $[0,t]$.

    \noindent (iv) The norm can be bounded using the triangle inequality with (i) and (iii).
\end{proof}
In the theorem above, we do not require $u' = Au$ or the associated deterministic Euler dynamics to be stable, i.e., $A$ does not need to to satisfy the conditions from Theorem~\ref{thm:lin_stab_fo}. However, if those conditions are satisfied, the quantities bounded in (i), (ii) converge to $0$, as $t \rightarrow \infty$ even if $h$ is constant.

Throughout this section, we have seen similarities of the deterministic Euler dynamics and the usual forward Euler method: we especially see the same global convergence order $O(h; h \downarrow 0)$ and at least a qualitative similarity in the close-to-instability regime of forward Euler methods. The stability theory of deterministic Euler methods yields rather different results. Whilst deterministic and stochastic Euler dynamics have closely related infinitesimal operators, they behave rather differently: smoothness and discontinuity of companion processes are a stark difference. 
In Section~\ref{Sec_Numerical}, we explore further differences and similarities between deterministic and stochastic Euler dynamics in numerical experiments. Until then, we focus on the stochastic Euler dynamics.

\section{Convergence and error analysis} \label{Sec_Converg}
We now study the convergence of the stochastic Euler dynamics $\tgn{V(t)}$ towards the underlying ODE $\tgn{u(t)}$ in the small stepsize limit $h \downarrow 0$. We illustrate this convergence behaviour in Figure~\ref{fig:convergence_SED}.
\begin{figure}
    \centering
    \includegraphics[scale = 0.77]{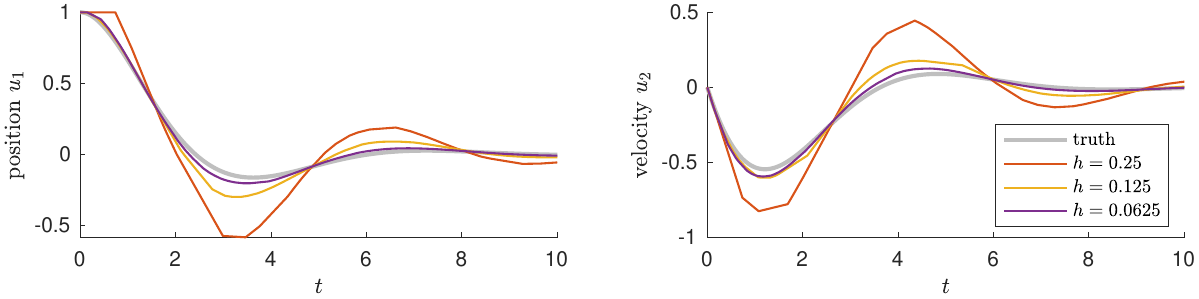}
    \caption{The trajectory of an underdamped harmonic oscillator $u_1' = u_2, u_2' = -u_1-u_2$ with initial value $u(0) = (1,0)^T$ and realisations of the corresponding stochastic Euler dynamics $\tgn{V(t)}$ for $h \in \{0.25, 0.125, 0.0625\}$. We show the ground truth $\tgn{u(t)}$  and one realisation of $\tgn{V(t)}$ for each of the stepsize parameters. The stochastic Euler dynamics approaches the ODE solution as $h$ decreases.}
    \label{fig:convergence_SED}
\end{figure}
Initially, we consider the convergence of the complete paths of the processes on the space $C := C^0([0,\infty); X)$, which we equip with the weighted metric $$d_C(\tgn{x(t)}, \tgn{y(t)}) = \int_0^\infty \exp(-t)\min\left\lbrace 1,\sup_{s\leq t}\|x(s)-y(s)\|\right\rbrace\mathrm{d}t \qquad (\tgn{x(t)}, \tgn{y(t)} \in C).$$ When we speak of weak convergence of a sequence of stochastic processes $(x_n(t))_{t \geq 0, n \in \mathbb{N}}: \Omega \rightarrow C^{\mathbb{N}}$ to a stochastic process $\tgn{x(t)}: \Omega \rightarrow C$, as $n\rightarrow \infty$, we mean that $$\lim_{n \rightarrow \infty}\mathbb{E}[F(\tgn{x_n(t)})] = \mathbb{E}[F(\tgn{x(t)})]$$ for any bounded and continuous $F: C \rightarrow \mathbb{R}$ and write $\tgn{x_n(t)} \Rightarrow \tgn{x(t)}$, as $n \rightarrow \infty$.
In Subsection~\ref{Subs_uni_bound}, we discuss the case of ODEs \eqref{eq_ODE} with bounded $f$, which is a setting that is slightly easier to handle. We then generalise these results to Lipschitz continuous $f$ in Subsection~\ref{Subs_uni_Lip}.

These convergence results do not lead to immediate error estimators. Thus, we follow a different approach in Subsection~\ref{subse_truncationerror}: we discuss and derive the local root mean square truncation error of the stochastic Euler dynamics.

\subsection{Uniform approximation for bounded $f$} \label{Subs_uni_bound}

Before we state the convergence theorem in the case of bounded $f$, we introduce an auxiliary stochastic process that allows us later to use a certain proof technique. Indeed, we define the \emph{jump time process} $(\overline{T}(t))_{t \geq 0}$ that records the jump times of $\tgn{V(t), \overline{V}(t)}$, i.e. $(\overline{T}(t))_{t \geq 0}$ is a Markov pure jump process and has the same  jump times $T_1, T_2,\ldots$ as $\tgn{V(t), \overline{V}(t)}$. At each jump time, it jumps to the value of that jump time $\overline{T}(T_k) := T_k$ $(k \in \mathbb{N}_0)$.

\begin{proposition}\label{prop_unif}
    Let $f:X \rightarrow X$ be Lipschitz continuous and bounded. Moreover, let $\tgn{u(t)}$ and $\tgn{V(t), \overline{V}(t)}$ be the ODE solution \eqref{eq_ODE} and the stochastic Euler dynamics \eqref{eq_SED} with initial value $u_0 \in X$, respectively. In addition, we assume to have access to the jump time process $(\overline{T}(t))_{t \geq 0}$. Then, $(V(t))_{t \geq 0} \Rightarrow (u(t))_{t \geq 0}\  (h \downarrow 0).$
\end{proposition}

At first sight, this convergence result appears to be much stronger than results we usually prove about numerical algorithms, where we do not necessarily formulate convergence for the whole positive real line. However, we equip the space $C$ with the metric $d_C$ that puts less weight on later points in time. We can interpret the results as  $(V(t))_{t \in [0,T]} \Rightarrow (u(t))_{t \in [0,T]}$ in $C^0([0,T]; X)$ with the usual uniform metric for any $T > 0$. Our technique does not provide us with a convergence rate.

The proof of Proposition~\ref{prop_unif} uses the perturbed test function methodology proposed by Kushner \cite{kushnerold}. The intuitive idea is that for a process $\tgn{V(t)}$ that is tight with respect to $h>0$, we can show convergence to $(u(t))_{t \geq 0}$ if we see convergence of the associated infinitesimal generators $\mathcal{A}_h$ to $\mathcal{A}$ as $h\downarrow 0$. Convergence of the infinitesimal generator is in the following sense: Let $T > 0$ and $\psi \in C^1_c(X)$ be a continuously differentiable test function with compact support and let $\Tilde{\psi}_h: [0, \infty) \times \Omega \rightarrow \mathbb{R}$ be a  \emph{perturbation} of this test function with
$$
\sup_{t \in [0,T], h \in (0, \varepsilon)}\mathbb{E}[|\Tilde{\psi}_h(t)|] < \infty, \qquad \qquad \lim_{h\rightarrow 0}\mathbb{E}[|\Tilde{\psi}_h(t)- \psi(V(t))|] =0 \quad (t \in [0,T]),
$$
for some $\varepsilon > 0$ small enough.
We then define the perturbed infinitesimal operator applied to $\Tilde{\psi}_h$ by $\tilde{\mathcal{A}}_h\tilde{\psi}_{h}$ through 
\begin{align*}
    \lim_{\tau \downarrow 0} \mathbb{E}\left[ \tilde{\mathcal{A}}_h\tilde{\psi}_{h}(t) -  \frac{\mathbb{E}\left[ \tilde{\psi}_{h}(t+\tau) \mid \overline{V}(s) : s \leq t\right] -  \tilde{\psi}_{h}(t)}{\tau}\right] = 0 \qquad (t \in [0,T]),
\end{align*}
and
\begin{align*}
    \sup_{\tau \in [0,T]} \mathbb{E}\left[ \tilde{\mathcal{A}}_h\tilde{\psi}_{h}(t) -  \frac{\mathbb{E}\left[ \tilde{\psi}_{h}(t+\tau) \mid \overline{V}(s) : s \leq t\right] -  \tilde{\psi}_{h}(t)}{\tau}\right] < \infty \qquad (t \in [0,T]).
\end{align*} Convergence is then implied by
\begin{equation}
    \sup_{t \in [0,T], h \in (0,\varepsilon)}\mathbb{E}[|\tilde{\mathcal{A}}_h\Tilde{\psi}_h(t)|] < \infty, \qquad \qquad \lim_{h\rightarrow 0}\mathbb{E}[|\tilde{\mathcal{A}}_h\Tilde{\psi}_h(t)- \mathcal{A}\psi(V(t))|] =0 \quad (t \in [0,T]), \label{eq_pertu_test}
\end{equation}
for some $\varepsilon > 0$ small enough.
If tightness and the statement \eqref{eq_pertu_test} hold for an appropriate perturbation of all $\psi \in C_c^1(X)$,   Theorem 3.2 in \cite{kushnerold} implies that Proposition~\ref{prop_unif} holds. We formulate these two statements as lemmas and prove them below.

\begin{lemma}\label{Lemma_tight} Let the assumptions of Proposition~\ref{prop_unif} hold. Then, $(V(t))_{t \geq 0}$ is tight with respect to $h > 0$.
\end{lemma}
\begin{proof}
    By definition, $\tgn{V(t)}$ is piecewise linear with Lipschitz constant bounded by $\sup_{v \in X} |f(v)|$. This supremum is finite as $f$ is bounded by definition. Thus, the tightness of $\tgn{V(t)}$ is implied by the Arzel\`a--Ascoli theorem, see, e.g. Theorem 7.3 in \cite{Billingsley}.
\end{proof}

\begin{lemma}\label{Lemma_pert} Let the assumptions of Proposition~\ref{prop_unif} hold. Then, there is a perturbation $\tilde{\psi}_h:[0,\infty) \times \Omega \rightarrow \mathbb{R}$ of every $\psi \in C_c^1(X)$ such that $\psi, \tilde{\psi}_h$ satisfy \eqref{eq_pertu_test}.
\end{lemma}
\begin{proof}
    Let $T > 0$ and $\psi \in C^1_c(X)$ be a suitable test function with compact support and globally bounded gradient. We now define the perturbed test function $\tilde{\psi}_{h}:=\psi(V(\cdot)) + \tilde{\psi}_{1,h}$, where  $\tilde{\psi}_{1,h}: [0,T]\rightarrow \mathbb{R}$ is given by
$$
\tilde{\psi}_{1,h}(t) = \int_t^T \exp\left(-\frac{s-t}{h}\right)\left\langle \nabla \psi(V(s)), f(\overline{V}(s))-f(V(s)) \right\rangle \mathrm{d}s.
$$
The choice of this perturbation was motivated by an idea mentioned  in Section 4.5 of \cite{kushnerold}. We first show that the perturbation disappears in mean as $h\downarrow 0$. We define the supremum $M_1:=\sup_{(v,\overline{v}) \in X^2}|\left\langle \nabla \psi(v), f(\overline{v})-f(v) \right\rangle|$, which is well-defined as $\psi$ has compact support and as $f$ is bounded. Then,
$$
|\tilde{\psi}_{1,h}(t)| \leq M_1 \int_0^T\exp(-s/h) \mathrm{d}s = h M_1(1-\exp(-T/h)).
$$
Thus,  $\mathbb{E}[|\tilde{\psi}_{1,h}(t)|]\rightarrow 0$ as $h \downarrow 0$. The same bound also implies that $\sup_{t \in [0,T], h \in (0, \varepsilon)}\mathbb{E}[|\Tilde{\psi}_h(t)|] < \infty$. Now we study the action of $\tilde{\mathcal{A}}_h$ on the perturbation.
There, we obtain
\begin{align*}
    \tilde{\mathcal{A}}_h\tilde{\psi}_{1,h}(t) = \lim_{\tau \downarrow 0} \frac{\mathbb{E}\left[ \tilde{\psi}_{1,h}(t+\tau) \mid (\overline{V}(s), \overline{T}(s)) : s \leq t\right] -  \tilde{\psi}_{1,h}(t)}{\tau} = \tilde{\psi}_{1,h}'(t) \qquad (t \geq 0),
\end{align*}
where the last equality follows as access to $\{(\overline{V}(s), \overline{T}(s)): s \leq t\}$ allows us to recover $V(t)$. Indeed, we have $V(t) = \overline{V}(t)+ (t-\overline{T}(t))f(\overline{V}(t))$.
We have 
\begin{align*}
   \tilde{\psi}_{1,h}'(t) &=  \underbrace{\int_t^T \frac{1}{h}  \exp\left(-\frac{s-t}{h}\right)\left\langle \nabla \psi(V(s)), f(\overline{V}(s))-f(V(s)) \right\rangle \mathrm{d}s}_{=:C_1} - \underbrace{\left\langle \nabla \psi(V(t)), f(\overline{V}(t))-f(V(t)) \right\rangle}_{=:C_2}.
\end{align*}
If now $\mathbb{E}[C_1] \rightarrow 0$ as $h \downarrow 0$, the term $C_2$ will allow us to replace $f(\overline{V}(t))$ by $f(V(t))$ in  $\tilde{\mathcal{A}}_h$ and, thus, show convergence to $\mathcal{A}$.
Defining $M_2:= \sup_{v \in X}\|\nabla \psi(v)\|$ and $L$ to be the Lipschitz constant of $f$, we have
\begin{align*}
    |C_1| &\leq  \int_t^T \frac{1}{h}  \exp\left(-\frac{s-t}{h}\right)\| \nabla \psi(V(s))\| \|f(\overline{V}(s))-f(V(s))\|  \mathrm{d}s \\
    &\leq M_2\sum_{k=0}^\infty \int_{T_k}^{T_{k+1}} \frac{1}{h}  \exp\left(-\frac{s}{h}\right)\|f(\overline{V}(s))-f(V(s))\|  \mathrm{d}s \\
    &= LM_2\sum_{k=0}^\infty \mathbf{1}[T_{k}\leq T] \int_{0}^{T_{k+1}-T_k} \frac{s}{h}  \exp\left(-\frac{s+T_k}{h}\right)  \mathrm{d}s \\ &= LM_2\sum_{k=0}^\infty \mathbf{1}[T_{k}\leq T] \exp\left(-\frac{T_k}{h}\right)\left( h - \exp\left(-\frac{T_{k+1}-T_k}{h}\right)(h+T_{k+1}-T_k)\right)
\end{align*}
Studying the limit $h \downarrow 0$ is non-trivial as the distribution of $T_{k+1}-T_k \sim \mathrm{Exp}(h^{-1})$ depends on $h$ itself. Thus, we write $h\Delta_k := T_{k+1}-T_k$, with $\Delta_k \sim \mathrm{Exp}(1)$.  Now, we have
\begin{align*}
     |C_1| &\leq LM_2\sum_{k=0}^\infty \exp\left(-\frac{T_k}{h}\right) \left( h - \exp\left(-\frac{h\Delta_k}{h}\right)h(1+\Delta_k)\right) \leq hLM_2\sum_{k=0}^\infty  \exp\left(-\frac{T_k}{h}\right) 
\end{align*}
We note that $T_k/h \sim \mathrm{Gamma}(k,1)$. By evaluating the moment-generating function of the Gamma distribution, we obtain
\begin{align*}
    \mathbb{E}[|C_1|] \leq hLM_2\sum_{k=0}^\infty 2^{-k},
\end{align*}
implying both
$\sup_{t \in [0,T], h \in (0, \varepsilon)}\mathbb{E}[|\tilde{\mathcal{A}}_h\Tilde{\psi}_h(t)- \mathcal{A}\psi(V(t))|] <\infty$
and 
\begin{align*}
 \lim_{h \downarrow 0} &\mathbb{E}[|\tilde{\mathcal{A}}_h \psi_h(t) - \mathcal{A}\psi(V(t))|] \\ &= \left\langle \nabla \psi(V(t)), f(\overline{V}(t)) \right\rangle -\left\langle \nabla \psi(V(t)), f(\overline{V}(t))-f(V(t)) \right\rangle + \left\langle \nabla \psi(V(t)), f(V(t)) \right\rangle = 0,
\end{align*}
proving the statement of the Lemma.
\end{proof}

The boundedness assumption in Proposition~\ref{prop_unif} is convenient, but probably too strong to be useful in practical applications. In the next subsection, we see how to generalise these results to unbounded, but Lipschitz continuous $f$.
\subsection{Uniform approximation for Lipschitz continuous $f$} \label{Subs_uni_Lip}
The theorem differs from Proposition~\ref{prop_unif} only in that it removes the boundedness assumption on $f$. 
\begin{theorem}\label{thm_unif_Lip}
    Let $f:X \rightarrow X$ be Lipschitz continuous and let $\tgn{u(t)}$ and $\tgn{V(t), \overline{V}(t)}$ be the associated ODE solution \eqref{eq_ODE} and stochastic Euler dynamics \eqref{eq_SED} with initial values $u_0 \in X$, respectively. Moreover, we assume that we have access to the jump time process $(\overline{T}(t))_{t \geq 0}$. Then, $(V(t))_{t \geq 0} \Rightarrow (u(t))_{t \geq 0}\  (h \downarrow 0).$
\end{theorem}

We usually show tightness of a family of processes by proving uniform equicontiuity of that family. An unbounded function $f$ makes it difficult to control the piecewise derivative of $\tgn{V(t)}$, which in turn makes it difficult to show equicontiuity. If tightness is difficult to show due to unboundness (as opposed to explosiveness) of a process, we can study a truncated process instead.  Let $R > 0$ and $f_R: X \rightarrow X$ be chosen such that $$f_R(v) = \begin{cases}
    f(v), &\text{if } \|v\| \leq R, \\
     0, &\text{if } \|v\| \geq R+1,
\end{cases} \qquad (v \in X / \{v: R < \|v\|< R+1\})$$
and $f_R$ being smooth on $\{v: R\leq \|v\|\leq R+1\}$. We define the truncated processes $\tgn{u_R(t)}$ and $\tgn{V_R(t), \overline{V}_R(t)}$ through \eqref{eq_ODE} and \eqref{eq_SED}, respectively, with $f_R$ replacing $f$. We use the idea
\begin{align*}
    \begin{matrix}
        \tgn{V_R(t)} & \xRightarrow{{\color{gray}(h \downarrow 0)}} & \tgn{u_R(t)} \vspace{0.2cm} \\ 
        {{\color{gray}\scriptstyle (R \rightarrow \infty)}} \Big\Downarrow & & \Big\Downarrow{{\color{gray}\scriptstyle (R \rightarrow \infty)}} \vspace{0.1cm} \\
        \tgn{V(t)} & \xRightarrow[{\color{gray}(h \downarrow 0)}]{\ } & \tgn{u(t)}
    \end{matrix}.
\end{align*}
to prove the convergence  as stated in Theorem~\ref{thm_unif_Lip}. Lemma~\ref{Lemma_tight} and \ref{Lemma_pert} are true for $\tgn{u_R(t)}$ and $\tgn{V_R(t), \overline{V}_R(t)}$ for all $R > 0$, and thus, imply $\tgn{V_R(t)} \Rightarrow \tgn{u_R(t)}$ as $h \downarrow 0$. The vertical convergence statements $(R \rightarrow \infty)$ follow from, e.g., the Corollary following Theorem~3.2 in \cite{kushnerold} or the discussion on pp.\ 160-161 in \cite{kushnernew}. These results then imply Theorem~\ref{thm_unif_Lip}. 
\subsection{Local root mean square truncation errors} \label{subse_truncationerror}
The convergence analysis in Proposition~\ref{prop_unif} and Theorem~\ref{thm_unif_Lip}  does not immediately lead to a mechanisms allowing us to estimate the error between ODE solution \eqref{eq_ODE} and stochastic Euler dynamics \eqref{eq_SED}. To get a sense of the order of convergence of the stochastic Euler dynamics, we now discuss the local truncation error that is achieved by the method.

Let $\varepsilon>0$ be small. The local truncation error of an ODE integrator is usually given as the difference between $u(\varepsilon)$ and a single timestep of the integrator with stepsize $\varepsilon$ in the limit as $\varepsilon \downarrow 0$. In the context of stochastic Euler dynamics, this approach is difficult to apply. Let $T_1$ be still the waiting time until the first jump in $(V(t), \overline{V}(t))_{t \geq 0}$. Then, 
$V(\varepsilon) = u_0 + \varepsilon f(u_0)$  with probability $\mathbb{P}(T_1 > \varepsilon) = \exp(-\varepsilon/h)$. Thus, assuming that $f$ is  Lipschitz continuous, we obtain an estimator of the form $\|u(\varepsilon) - V(\varepsilon)\| = O(\varepsilon^2; \varepsilon \downarrow 0)$ with a certain probability. Now, as $\varepsilon$ approaches zero, the probability of the estimator being correct converges to $1$. Traditionally, however, we would let time point $\varepsilon$ and discretisation stepsize $h$ approach zero simultaneously; usually choosing $h = \varepsilon$. Then, however, $\mathbb{P}(T_1 > \varepsilon | h = \varepsilon) = 1/\mathrm{e} \not\approx 1$. 

In the following, we propose a different way to analyse the local truncation error of the stochastic Euler dynamics. Indeed, we compute the local root mean square truncation error $\mathbb{E}[\|V(\varepsilon) - u(\varepsilon)\|^2]^{1/2}$ in the case $f$ being linear and obtain a deterministic statement about the convergence rate -- this approach is different from the mean square error analysis given in \cite{abdulle}, where the error is considered for the discrete-time trajectory.
In our following derivation and analysis we choose a linear ODE \eqref{eq_ODE}, i.e., we let $f(u) = Au$ for some $A \in \mathbb{R}^{d \times d}$. Moreover, we need be able to keep track of how often the stochastic Euler dynamics jumps within a certain time interval. To this end, we define $(K(t))_{t \geq 0}$ to be the (homogeneous) Poisson (point) process with rate $1/h$ on $[0, \infty)$ that is coupled to $(V(t), \overline{V}(t))_{t \geq 0}$ by having identical jump times. In particular, $(K(t))_{t \geq 0}$ is the piecewise constant jump process on $\mathbb{N}_0$ that commences at $K(0) = 0$ and is incremented by one at every jump time $T_1, T_2,\ldots$ of $(V(t), \overline{V}(t))_{t \geq 0}$. Indeed, if $K(\varepsilon) = k \in \mathbb{N}_0$, we have $T_k \leq \varepsilon$ and $T_{k+1} > \varepsilon.$ Thus, if we know $K(\varepsilon)$, we know exactly how many jumps happened up to time $\varepsilon > 0$. Conditionally on $K(\varepsilon)$ the distribution of $V(\varepsilon)$ has a specific form:
$$\mathbb{P}(V(\varepsilon) \in \cdot | K(\varepsilon) = k) = \mathbb{P}((I_d + (\varepsilon - T_k) A)(I_d + H_k A)\cdots(I_d + H_1A)u_0 \in \cdot | K(\varepsilon) = k).$$

Next, we aim to  compute the local root mean square truncation error conditioned on $K(\varepsilon) = k$. The (unconditioned) local root mean square truncation error can then be obtained through the law of total expectation. We first determine the distribution of the random timestep sizes $H_1,\ldots, H_k$ given that $K(\varepsilon) = k$. This is a generalisation of a well-known result about exponential distributions.

\begin{lemma} \label{lemma_uniform}
Let $T_1,T_2,\ldots$ be the jump times of the Poisson point process $(K(t))_{t \geq 0}$, $T_0 = 0$, and $H_k = T_k-T_{k-1} \sim \mathrm{Exp}(h^{-1})$, $k \in \mathbb{N}$, as before. Then,
$$
\mathbb{P}((H_1,\ldots,H_k) \in \cdot | K(t) = k) = \mathrm{Unif}(S_{k,t}) \qquad (k \in \mathbb{N}, t \geq 0),$$
where $S_{k,t} := \{(h_1,\ldots,h_k)\in [0,\infty)^k: h_1+\cdots+h_k \leq t\}$ 
is a $k$-dimensional simplex.
\end{lemma}
\begin{proof}
    We use Bayes' formula to find the probability density function of $\mathbb{P}((H_1,\ldots,H_k) \in \cdot | K(t) = k)$ and, thus, prove the identity above. Bayes' formula states that 
    $$
   p_{H_1,\ldots,H_k| K(t) = k}(h_1,\ldots,h_k)  = \frac{p_{K(t)| H_1=h_1,\ldots,H_k = h_k}(k) p_{H_1,\ldots,H_k}(h_1,\ldots,h_k)}{p_{K(t)}(k)} \qquad (k \in \mathbb{N}, h_1,\ldots,h_k\geq 0),
    $$
    where the $p$'s describe probability density functions with respect to Lebesgue or counting measures, as appropriate. Let again $k \in \mathbb{N}$, and $h_1,\ldots,h_k\geq 0.$
    We start by noting that
\begin{align*}
    p_{H_1,\ldots,H_k}(h_1,\ldots,h_k) = \frac{1}{h^{k}}\exp\left(-\frac{h_1+\cdots + h_K}{h}\right),  \qquad 
    p_{K(t)}(k) = \left(\frac{t}{h}\right)^k\frac{\exp(-t/h)}{k!},
\end{align*}
by definition and by the property $K(t) \sim \mathrm{Poisson}(t/h)$ $(t \geq 0)$ of Poisson processes, respectively. The conditional density $p_{K(t)| H_1=h_1,\ldots,H_k = h_k}$ describes the event that $H_{k+1}>t-h_1-\cdots - h_k$ whilst requiring that $h_1+\cdots + h_k \leq t$. Thus, we have 
$$
p_{K(t)| H_1=h_1,\ldots,H_k = h_k}(k) = \exp\left(\frac{-t+h_1+\cdots+h_k}{h}\right)\mathbf{1}[h_1+\cdots + h_k \leq t]
$$
and finally
\begin{align*}
       p_{H_1,\ldots,H_k| K(t) = k}(h_1,\ldots,h_k)  &= \frac{\exp\left(\frac{-t+h_1+\cdots+h_k}{h}\right)\mathbf{1}[h_1+\cdots + h_k \leq t]\frac{1}{h^{k}}\exp\left(-\frac{h_1+\cdots + h_K}{h}\right)}{\left(\frac{t}{h}\right)^k\frac{\exp(-t/h)}{k!}} \\
       &= \frac{k!}{t^k} \mathbf{1}[h_1+\cdots + h_k \leq t],
\end{align*}
which is the correctly normalised probability density function of $\mathrm{Unif}(S_{k,t})$, see, e.g., \cite{Simplex}.
\end{proof}
There is an interesting connection between this result and the random Euler methods discussed in \cite{Eisenmann, JENTZEN2009346}. There, the random Euler method is initialised with a set of deterministic discretisation points $t_0, t_1,...$. Then $J \in \mathbb{N}$ random timepoints are sampled uniformly and i.i.d. from  $[t_{k}, t_{k+1}]$ $(k \in \mathbb{N}_0)$ and used for a Monte Carlo approximation of the integral formulation of the differential equation. In Lemma~\ref{lemma_uniform}, we show that we implicitly sample uniformly when we are on a fixed time interval $[t_{k}, t_{k+1}]$ with a fixed number of samples $K(t_{k+1}- t_{k}) = J$, but on the simplex $S_{J,(t_{k+1}- t_{k})}$ instead of the hypercube $[0,1]^J$ (as in the i.i.d. case).

We can now estimate the local root mean square truncation error of $(V(t))_{t \geq 0}$ with respect to the ODE solution $(u(t))_{t \geq 0}$.

\begin{theorem} \label{thm:RMSTE}
    Let $(V(t), \overline{V}(t))_{t \geq 0}$ be the stochastic Euler dynamics \eqref{eq_SED} corresponding to $u' = Au$, $u(0) = u_0$, with $A \in \mathbb{R}^{d \times d}$, and let $(K(t))_{t \geq 0}$ be the associated Poisson process. Then, 
    \begin{enumerate}
    \item[(i)] $\mathbb{E}[\|V(\varepsilon)- u(\varepsilon)\|^2| K(\varepsilon) = k]^{1/2}  =  O(\varepsilon^2; \varepsilon \downarrow 0)$ $(k \in \mathbb{N}_0)$,
        
        \item[(ii)] $\mathbb{E}[\|V(\varepsilon)- u(\varepsilon)\|^2]^{1/2}  =  O(\varepsilon^2; \varepsilon \downarrow 0)$.
    \end{enumerate}
\end{theorem}
\begin{proof} We aim to show a second order local truncation error. Thus, in the following, we can replace $u(\varepsilon) = \exp(\varepsilon A)u_0$ by its first-order Taylor approximation $u_0 + \varepsilon Au_0$.

\noindent    (i) Let $k \in \mathbb{N}_0$. 
    We write $P_k := (I_d + h_k A)\cdots(I_d + h_1A).$ Then,
    \begin{align*}
       \mathbb{E}[V(\varepsilon)|K(\varepsilon) = k]  &= \frac{k!}{\varepsilon^k}\int_{S_{k,\varepsilon}} (I_d + (\varepsilon - h_1-\cdots-h_k) A)P_k u_0\mathrm{d}(h_1,\ldots,h_k) \\
       &= \frac{k!}{\varepsilon^k}\int_{S_{k,\varepsilon}} (I_d + \varepsilon A)P_k u_0  - (h_1+\cdots+h_k) AP_ku_0 \mathrm{d}(h_1,\ldots,h_k) \\
       &=  u_0 + \varepsilon A u_0 + \frac{k!}{\varepsilon^k}\int_{S_{k,\varepsilon}} \left((I_d + \varepsilon A)(P_k-I_d)  - (h_1+\cdots+h_k) AP_k\right) u_0\mathrm{d}(h_1,\ldots,h_k).
    \end{align*}
    We first see that the integrand is a polynomial with constant term $=0$.
    If we study the terms in the integrand above that are linear in $h_1,\ldots,h_k$, we see that 
    $$
   (I_d + \varepsilon A)( h_kA + \cdots + h_1A)u_0 - (h_1 + \cdots + h_k)Au_0 = \varepsilon (h_1 + \cdots + h_k)A^2u_0.
    $$
    Thus, the integrand above can be written as a sum of monomials with the lowest order being 1.
 If we assume that $\varepsilon \in (0,1)$, we can find a constant $C_k >0$ such that
 \begin{align*}
     \mathbb{E}&[\|V(\varepsilon)- u_0 - \varepsilon A u_0 \|^2|K(\varepsilon) = k] \leq \frac{k!}{\varepsilon^k}\int_{S_{k,\varepsilon}} C_k \varepsilon^2\sum_{i,j=1}^k h_ih_j\mathrm{d}(h_1,\ldots,h_k) = O(\varepsilon^4; \varepsilon \downarrow 0),
 \end{align*}
  where we obtain the order by bounding $\sum_{i,j=1}^k h_ih_j \leq \varepsilon^2$.

   \noindent (ii) Our goal is to employ the law of total expectation
    $$
     \mathbb{E}[\|V(\varepsilon)- u_0 - \varepsilon A u_0\|^2] = \sum_{k=0}^\infty  \mathbb{E}[\|V(\varepsilon)- u_0 - \varepsilon A u_0\|^2|K(\varepsilon) = k] \mathbb{P}(K(\varepsilon) = k).
    $$
 To employ the bound from (i), we need to show that the $C_k>0$  do not diverge too quickly as $k \rightarrow \infty$. Under the hypothesis that $K(\varepsilon) = k$, we have 
    \begin{align*}
       \|V(\varepsilon)- u_0 - \varepsilon A u_0\|^2 &\leq \|I_d + (\varepsilon - H_1- \cdots - H_k)A\|^2 \prod_{i = 1}^k\|I_d + H_i A\|^2\|u_0\|^2 \\ &\leq  (1 + \varepsilon\|A\|)^2 \prod_{i = 1}^k(1 + H_i \|A\|)^2\|u_0\|^2.
    \end{align*}
    $C_k$ is chosen in (i) such that it bounds the sum of all constant factors that appear in the polynomial $$(h_1,\ldots,h_k) \mapsto\|V(\varepsilon)- u_0 - \varepsilon A u_0\|^2 =  \|\left((I_d + \varepsilon A)(P_k-I_d)  - (h_1+\cdots+h_k) AP_k\right) u_0\|^2.$$
   We obtain the sum of constant factors in the polynomial by evaluating it for $(1,\ldots,1)$ and then choose
   $$
   C_k := (1+\|A\|)^{2k+2} \|u_0\|^2.
   $$
   as an upper bound.
    Now, we have 
    \begin{align*}
         \mathbb{E}[\|V(\varepsilon)- u_0 - \varepsilon A u_0\|^2] &\leq \sum_{k=0}^\infty (1+\|A\|)^{2k+2} \varepsilon^{4} \mathbb{P}(K(\varepsilon) = k)
        \\ &= \varepsilon^{4} (1+\|A\|)^{2} \sum_{k=0}^\infty  \left(\frac{\varepsilon(1+\|A\|)^{2}}{h}\right)^k \frac{\exp(-\varepsilon/h)}{k!}
        \\ &= \varepsilon^{4} (1+\|A\|)^{2}\exp((2\varepsilon \|A\|+\varepsilon\|A\|^2)/h)
    \end{align*}
    and, thus, proven (ii).
\end{proof}
One motivation for the local root mean square truncation error has been the behaviour of the estimator when $h\downarrow 0$.
In the end of the proof of Theorem~\ref{thm:RMSTE}(ii), we develop the error bound
 \begin{align*}
         \mathbb{E}[\|V(\varepsilon)- u(\varepsilon)\|^2] &\leq  \varepsilon^{4} (1+\|A\|)^{2}\exp((2\varepsilon \|A\|+\varepsilon\|A\|^2)/h).
    \end{align*}
   Here, if $h \downarrow 0$, the bound diverges, which is in slight contrast to the (weaker) convergence results obtained in Proposition~\ref{prop_unif} and Theorem~\ref{thm_unif_Lip}. However, if we choose $h = O(\varepsilon; \varepsilon \downarrow 0)$, we still obtain convergence in the described order -- an improvement over our preliminary discussion of truncation errors for stochastic Euler dynamics.

We finally note that the root mean square errors in Theorem~\ref{thm:RMSTE} also allow us to estimate the variance of $V(\varepsilon)$ using the simple inequality
$$
\mathbb{E}[\|V(\varepsilon)- \mathbb{E}[V(\varepsilon)]\|^2] \leq \mathbb{E}[\|V(\varepsilon)- u(\varepsilon)\|^2]  =  O(\varepsilon^4; \varepsilon \downarrow 0).
$$

\section{Stability} \label{Sec_Stab}
We now study the stability of the stochastic Euler dynamics, i.e., we aim to show for certain examples that the asymptotic behaviour of $(V(t))_{t \geq 0}$ for certain $h > 0$ resembles the asymptotic behaviour of $\tgn{u(t)}$. We start with an intuitive discussion of the stability of the random timestep Euler method $(\widehat{V}_k)_{k=1}^\infty$ in Subsection~\ref{subsec_stab_JChain}. We then move on to studying stability for the stochastic Euler dynamics in terms of Foster--Lyapunov criteria in Subsections~\ref{Subs_FostLya_1D}, which also provide us with an asymptotic rate.

\subsection{Stability of the random timestep Euler method} \label{subsec_stab_JChain}
Let $a > 0$ and let $(u(t))_{t \geq 0}$ satisfy the linear ODE $\dot{u} = -au$ on $X = \mathbb{R}$. Similar to deterministic ODE solvers, we expect stability problems if the mean stepsize $h$ is chosen to be too large. Indeed, we are wondering how $ah$ needs to be chosen for $\hat{V}_k \rightarrow 0$, in an appropriate sense, as $k \rightarrow \infty$. We now derive conditions under which mean and second moment of the random timestep Euler method converge to zero as $k \rightarrow \infty$.

We first recall that $\hat{V}_{k+1} := (1-aH_k)\hat{V}_k$ ($k \in \mathbb{N}_0$)
for independent $H_1, H_2, \ldots \sim \mathrm{Exp}(\lambda)$ and obtain
$$
\mathbb{E}[\hat{V}_{k+1}] = \mathbb{E}[(1-aH_k)\hat{V}_{k}] = \mathbb{E}[1-aH_k]\mathbb{E}[\hat{V}_{k}] = (1-ah)\mathbb{E}[\hat{V}_{k}] = (1-ah)^{k+1}u_0.
$$
which converges geometrically to $0$ if $ah < 2$ -- and, thus, behaves identically to the forward Euler method with stepsize $h$. Next, we compute the second moment of the dynamical system:
$$
\mathbb{E}[\hat{V}_{k+1}^2] = \mathbb{E}[(1-aH_k)^2\hat{V}_{k}^2]  = \mathbb{E}[1-2aH_k+a^2H_k^2] \mathbb{E}[\hat{V}_{k}^2] = (1-2ah+ 2a^2h^2)^{k+1}u_0^2.
$$
Now $|1-2ah+ 2a^2h^2|<1$, if $ah<1$. Thus, we obtain a stricter stepsize restriction if also second moments are supposed to converge.  The results given here can be easily extended to more general linear ODEs $u' = -Au$, with symmetric positive definite $A$ -- we give a detailed explanation in the proof of Theorem~\ref{thm_stab_higherdim}.

Intuitively, stability of the random timestep Euler method should also imply stability of the stochastic Euler dynamics. Indeed, we obtain the same convergence criteria below when studying Foster--Lyapunov criteria. The Foster--Lyapunov criteria, however, additionally supply us with a continuous convergence rate of the dynamics to their stationary state. We can also see the convergence rate $|1-2ah+ 2a^2h^2|$ in the analysis above. However, $\mathbb{E}[\hat{V}^2_{k}]$ does not relate to a particular time $t \geq 0$. Thus, the rate above does not allow us to compare the convergence-to-stationarity behaviour of $(V(t))_{t \geq 0}$ to that of $(u(t))_{t \geq 0}$.
\subsection{Foster--Lyapunov criteria in the case of linear ODEs} \label{Subs_FostLya_1D}

We begin our study of Foster--Lyapunov criteria for the stochastic Euler dynamics by considering the the linear ODE $u' = -au$ on $X = \mathbb{R}$ for some $a > 0$. We are later able to use this one-dimensional result to study the higher dimensional version. The first step is to find an appropriate Lyapunov function for $(V(t),\overline{V}(t))_{t \geq 0}$. We consider a Lyapunov function with the following structure $$L(v,\overline{v}) := c_1v^2 + c_2\overline{v}^2 + c_3(v-\overline{v})^2 
\qquad (v, \overline{v} \in X),$$
with some constants $c_1,c_2,c_3 \geq 0$. When computing the action of $\mathcal{A}_h$ with respect to this Lyapunov function, we obtain
\begin{align*}
    \mathcal{A}_h L(v,\overline{v}) &= -a\left(\frac{\partial}{\partial v} L(v,\overline{v})\right)\overline{v} + (L(v,v) - L(v,\overline{v}))/h \\
    &= (c_2 - c_3)v^2/h  + (-c_2/h-(1/h +2a)c_3) \overline{v}^2  + 2((1/h-a)c_3 -ac_1)v\overline{v} 
\end{align*}
In the following lemma, we show that we can bound $\mathcal{A}_h L$ in terms of $-\kappa L$. This bound then implies that $\mathbb{E}[L(V(t),\overline{V}(t))]$ converges to zero at exponential speed with rate $- \kappa$. We state and prove this exponential bound formally in a proposition just below the lemma.
\begin{lemma} \label{Lemma_Lyap}Let $\mathcal{A}_h$ be the infinitesimal generator of the stochastic Euler dynamics \eqref{eq_SED} corresponding to the ODE $u' = -au$, with $a > 0$ and $L$ be the function introduced above. Moreover, we assume that $h > 0$ is chosen such that $ah <1$. Then,
$$
\mathcal{A}_h L(v,\overline{v}) \leq -\kappa L(v,\overline{v}) \qquad (v,\overline{v} \in X),
$$
with $c_1 = 1, c_2 = 0$, $c_3 := 1/\max\{(1/h -\kappa)/{\kappa},(\kappa-1/h+a)/a \}$, and $\kappa  \in (0,\min\{2a,1/(2h)\})$.
\end{lemma}
\begin{proof}    
    We aim to show that there is some $\kappa > 0$ with
\begin{align*}
       \begin{cases}
      c_2/h - c_3/h  &\leq  -\kappa c_1-\kappa c_3 \\ 
       -c_2/h-(1/h +2a)c_3  &\leq  -\kappa c_2 -\kappa dc_3 \\
        (1/h-a)c_3 -ac_1 &\leq \kappa  c_3
    \end{cases} \quad  \Leftrightarrow \quad
       \begin{cases}
       \kappa  c_1/c_3+c_2/(c_3h)  &\leq 1/h-\kappa  \\ 
        (\kappa -1/h) c_2/c_3 &\leq  1/h +2a-\kappa  \\
         c_1/c_3 &\geq (\kappa -1/h+a)/a.
    \end{cases}
\end{align*}
The inequality in the top right  implies that $\kappa  < h^{-1}$.
Assuming that $\kappa  < 2a$, we have
\begin{align*}
       \begin{cases}
       \kappa  c_1/c_3+ c_2/(c_3h)  &\leq  1/h-\kappa  \\ 
        c_2/c_3 &\geq 0 > \frac{1/h +2a-\kappa }{\kappa -1/h} \\
         c_1/c_3 &\geq (\kappa -1/h+a)/a.
    \end{cases}
\end{align*}
 Now, we can certainly choose $c_2/c_3 := 0$ and $c_1/c_3 := \max\{(1/h -\kappa )/{\kappa },(\kappa -1/h+a)/a \}$. If $(1/h -\kappa )/{\kappa } \geq (\kappa -1/h+a)/a$, all of the inequalities are already satisfied. If otherwise, $(1/h -\kappa )/{\kappa } < (\kappa -1/h+a)/a$, we need to  make sure that 
 $$
\kappa^2-\kappa/h+a\kappa  \leq a/h-a\kappa  \Leftrightarrow 
 -\sqrt{a^2 + 1/(2h)^2}- a +1/(2h)\leq \kappa  \leq \sqrt{a^2 + 1/(2h)^2} - a+1/(2h).
 $$
 which is certainly satisfied if $\kappa  \in (0,1/(2h))$.
\end{proof}

\begin{proposition} \label{prop_con er}
Let $(V(t), \overline{V}(t))_{t \geq 0}$ be the stochastic Euler dynamics \eqref{eq_SED} corresponding to the ODE $u' = -au$, with $a > 0$. Moreover, let $ah<1$, $\kappa  \in (0, \min\{2a,1/(2h)\})$, and $c_3 := 1/\max\{(1/h -\kappa )/{\kappa },(\kappa -1/h+a)/a \}$. Then, 
$$
\mathbb{E}\left[V(t)^2\right] + {c_3}\mathbb{E}\left[(V(t)-\overline{V}(t))^2\right] \leq \exp(-\kappa  t)u_0^2.
$$
\end{proposition}
\begin{proof}
    If we apply \cite[Theorem 1.1]{MeynTweedieIII} to the bound shown in Lemma~\ref{Lemma_Lyap}, we obtain 
    $$ \mathbb{E}[V(t)^2 + c_3 (V(t)-\overline{V}(t))^2]\leq - \kappa \int_0^t\mathbb{E}[V(t)^2 + c_3 (V(t)-\overline{V}(t))^2]\mathrm{d}s +u_0^2 + c_3(u_0 - u_0)^2. $$
An application of the Gr\"onwall inequality then implies the statement of the proposition.
\end{proof}

This proposition implies that if we choose $h$ to be small (namely $h \in (0, 1/(4a))$), the speed of convergence to stationarity of the stochastic Euler dynamics is at most infinitesimal smaller than that of the original dynamical system $(u(t))_{t \geq 0}$. If, on the other hand, $h \in (1/(4a), 1/a)$, we obtain a slower rate that depends on $h$. This is surprising as the forward Euler method for $u'=-au$ actually converges faster for stepsizes close to $1/a$. From the discussion in Subsection~\ref{subsec_stab_JChain}, we would not expect stability of the second moment of $(V(t))_{t \geq 0}$, if $h > 1/a$.

Based on Proposition~\ref{prop_con er}, we can now study the stability of higher dimensional linear ODEs.

\begin{theorem} \label{thm_stab_higherdim}
    Let $A \in \mathbb{R}^{d \times d}$ be symmetric positive definite and let $(V(t),\overline{V}(t))_{t \geq 0}$ be the stochastic Euler dynamics \eqref{eq_SED} corresponding to $u' = -Au$. Let $\lambda_1\leq \cdots \leq \lambda_d$ be the eigenvalues of $A$,let $h$ be chosen such that $\lambda_{d} h < 1$, and let $\kappa'   \in (0, \min\{2\lambda_{1}, 1/(2h)\})$. Then, 
    $$
    \mathbb{E}\left[\|V(t)\|^2\right]  + c_3'\mathbb{E}[\|V(t) - \overline{V}(t)\|^2] \leq \exp(-\kappa' t)\|u_0\|^2,
    $$
    with $c_3' := \min\{ 1/\max\{(1/h -\kappa )/{\kappa },(\kappa -1/h+\lambda_\ell)/\lambda_\ell \}: \ell \in \{1,\ldots,d\}\}$.
\end{theorem}
\begin{proof}
 Since $A$ is symmetric, there is a similarity transform $A = S\Lambda S^T$, where $S$ is an orthonormal matrix consisting of the eigenvectors of $A$ and $\Lambda$ is the diagonal matrix containing the associated real eigenvalues on the diagonal. Then, we can write $V(t)$ under the condition that $K(t) = k$ as 
 \begin{align*}
(I_d -(t-H_1-\cdots - &H_k)A) (I_d -H_k A)\cdots(I_d -H_1A) u_0 \\&=  S(I_d -(t-H_1-\cdots - H_k)\Lambda)(I_d -H_k\Lambda)\cdots(I_d -H_1\Lambda)S^Tu_0,
 \end{align*}
 for $t \geq 0.$ Thus, we can assume that $A := \Lambda$: we just measure convergence in $\|S \cdot\| = \|\cdot \|$ and consider the initial value $S^{T}u_0$ that satisfies $\|S^{T}u_0\| = \|u_0 \|$. Then, however, we have $d$ processes that are only coupled through their jump times. If we apply Proposition~\ref{prop_con er} to all $d$ components of $V(t) = (V_1(t),\ldots,V_d(t))$ and $\overline{V}(t)= (\overline{V}_1(t),\ldots,\overline{V}_d(t))$ with the stepsize restriction $\lambda_{d} h < 1$, we obtain 
\begin{equation} \label{eq:ineq_Vell}
    \sum_{\ell =1}^d \left(\mathbb{E}\left[V_\ell(t)^2\right] + {c_3^{(\ell)}}\mathbb{E}\left[(V_\ell(t)-\overline{V}_\ell(t))^2\right] \right)\leq  \sum_{\ell =1}^d \exp(-\kappa_\ell  t)u_{0,\ell}^2
\end{equation}

 where $(c_3^{(1)},\ldots, c_3^{(d)})$ and $(\kappa^{(1)},\ldots, \kappa^{(d)})$ are the constants we obtain when applying Proposition~\ref{prop_con er} to each of $(V_1(t), \overline{V}_1(t))_{t \geq 0},\ldots,(V_d(t), \overline{V}_d(t))_{t \geq 0}$. We denote $c_3' = \min \{c_3^{(1)},\ldots, c_3^{(d)}\}$ and $\kappa' = \min\{\kappa^{(1)},\ldots, \kappa^{(d)}\}$. Then, we can employ \eqref{eq:ineq_Vell} to obtain the asserted inequality.
\end{proof}

\section{Second-order stochastic Euler dynamics} \label{Sec_higherorder}
The forward Euler method is a first-order numerical integrator. There are several ways to construct ODE solvers with higher convergence orders. The explicit midpoint rule, for instance, is a  second-order method. It is defined through the following recursive relationship:
$$\hat{y}_k= \hat{y}_{k-1} + {h_k}f\left(\hat{y}_{k-1}+ \frac{h_k}{2} f(\hat{y}_{k-1})\right) \quad (k \in\mathbb{N}), \qquad \hat{y}_0 = u_0.$$
Of course, this method can be employed with a random timestep size. The derivation and especially the analysis of associated \emph{stochastic explicit midpoint dynamics} should now be much trickier as a natural companion process would not be piecewise constant with jumps but piecewise linear with jumps. We defer the treatment of this and other Runge--Kutta-based stochastic dynamics to future work and instead follow a different idea:

The stochastic Euler dynamics construct a piecewise linear path by using a first-order ODE with piecewise constant right-hand sides. We can similarly construct a piecewise polynomial path by constructing an ODE that has constant higher derivatives. For a second order dynamics, for instance, we can consider
$\ddot{x} = a_1, x(0) = a_2, \dot{x}(0) = a_3$ for constant vectors $a_1, a_2, a_3 \in X$, which yields $x(t) = a_1t^2/2 + a_2 + a_3 t$ $(t \geq 0)$. Hence, rather than having a piecewise constant first derivative as in the stochastic Euler dynamics, we obtain a piecewise quadratic path through a piecewise constant second derivative.

We now need to assume some additional smoothness in $(0, \infty) \ni t \mapsto u(t)$ to then choose $a_1, a_2, a_3$ by matching terms in the Maclaurin expansion of $(u(t))_{t \geq 0}$. For some $h_1 > 0$, we can write the Maclaurin expansion as
\begin{align*}
    u(h_1) &= u(0) + u'(0)h_1 + u''(0)\frac{h_1^2}{2} + u^{(3)}(0)\frac{h_1^3}{6} + \cdots,
\end{align*}
which is either an infinite sum or ends with a remainder term.  
We now compare the first three terms in the Maclaurin expansion with the polynomial above and obtain 
$$
a_1 = u''(0), \qquad a_2 = u(0), \qquad a_3 = u'(0).
$$
Whilst we are usually able to evaluate $u' = f(u)$, additional work is necessary to obtain the second derivative $u'' = \mathrm{J}f(u)\cdot f(u)$. In the following we assume to be able to access this second derivative. Then, with the previous discussion in mind, we  define the \emph{second-order stochastic Euler dynamics} $(Y_1(t), Y_2(t), \overline{Y}(t))_{t \geq 0}$ through
\begin{align*}
Y'_1(t) &= Y_2(t) &(t >0) \\
    Y'_2(t) &= \mathrm{J}f(\overline{Y}(t))\cdot f(\overline{Y}(t))  &(t \in (T_{k-1}, T_k], k \in \mathbb{N}) \\
    \overline{Y}'(t) &= 0  &(t \in (T_{k-1}, T_k], k \in \mathbb{N})\\
    Y_1(T_{k-1}) &= Y_1(T_{k-1}-) &(k \in \mathbb{N}) \\
    Y_2(T_{k-1}) &= Y_2(T_{k-1}-) &(k \in \mathbb{N}) \\
    \overline{Y}(T_{k-1}) &= Y_1(T_{k-1}-) &(k \in \mathbb{N}) \\
    Y_1(0) &= \overline{Y}(0) =  u_0 &\\
    Y_2(0) &= f(u_0) &
\end{align*}
where the $(T_k)_{k\in \mathbb{N}}$ still denote appropriate jump times with i.i.d. increments $H_1, H_2,\ldots \sim \mathrm{Exp}(1/h)$. Again, it is easy to see that $(Y_1(t), Y_2(t), \overline{Y}(t))_{t \geq 0}$ is a well-defined Feller process with infinitesimal generator 
$$
\mathcal{A}_h^2 \varphi(y_1,y_2,\overline{y}) = \left\langle \nabla_{u} \varphi(y_1,y_2,\overline{y}),\begin{pmatrix}
    y_2 \\ \mathrm{J}f(\overline{y})\cdot f(\overline{y})
\end{pmatrix}  \right\rangle  + h^{-1} (\varphi(y_1,y_2,y_1) - \varphi(y_1,y_2,\overline{y}))
$$
for an appropriate test function $\varphi: X^3 \rightarrow \mathbb{R}$. Since only the second derivative of the path $(Y_1(t))_{t \geq 0}$ is now piecewise constant,  $(Y_1(t))_{t \geq 0}$ turns out to be a continuously differentiable approximation of $(u(t))_{t \geq 0}$. Conceptually, the ansatz for our ODE solution is, thus, now a smooth chain of quadratic splines with random interpolation points. A deterministic quadratic spline ansatz for the solution of ODEs has, for instance, previously been discussed by \cite{splines}. The preceeding discussion is \emph{easily} extended to orders beyond two, at least in cases in which higher order derivatives of $u$ are available. Random timestep explicit Runge-Kutta methods, such as the explicit midpoint rule above, should be able to circumvent this necessity.

We finish this section by studying the stability of the jump chain associated with the second order stochastic Euler dynamics $(\hat{Y}_{k,1}, \hat{Y}_{k,2})_{k=0}^\infty$, given by $(\hat{Y}_{k,1},\hat{Y}_{k,2}) := (Y_1(T_k), Y_2(T_k))$ $(k \in \mathbb{N})$ for a one-dimensional linear ODE.  In Section~\ref{Sec_Numerical}, we numerically explore the local root mean square truncation error of the second-order stochastic Euler dynamics (and, thus, actually justify its name).

\subsection{Stability of the second order stochastic Euler dynamics jump chain}
Similarly to our discussion in Subsection~\ref{subsec_stab_JChain}, we now study the stability of the jump chain $(\hat{Y}_{k,1},\hat{Y}_{k,2})_{k=0}^\infty$ when defined to approximate the linear ODE $\dot{u} = -a u$ on $X = \mathbb{R}$ for some $a > 0$. We actually focus only on the first component $(\hat{Y}_{k,1})_{k=0}^\infty$ here. We have
$$
\hat{Y}_{k,1} = \hat{Y}_{k-1,1} - aH_k\hat{Y}_{k-1,1}  + \frac{a^2H_k^2}{2} \hat{Y}_{k-1,1} = \left(1 - aH_k + \frac{a^2H_k^2}{2}\right) \hat{Y}_{k-1,1}  \qquad (k \in \mathbb{N}),
$$
where $H_1, H_2, \ldots \sim \mathrm{Exp}(h^{-1})$ and  $\hat{Y}_0 = u_0.$
 We again compute mean and second moment of the jump chain and investigate their longtime behaviour. We have
$$
\mathbb{E}[\hat{Y}_{k}] = \mathbb{E}\left[1 - aH_k + \frac{a^2H_k^2}{2}\right] \mathbb{E}[\hat{Y}_{k-1}] = \left(1 - ah + a^2h^2
\right)^ku_0,
$$
which converges to $0$ as $k \rightarrow \infty$ iff $ah < 1$. This condition is stricter than the condition we obtained in Subsection~\ref{subsec_stab_JChain} regarding the mean of the random timestep Euler method. Regarding the second moment, we see that
\begin{align*}
    \mathbb{E}[\hat{Y}_{k}^2] = \mathbb{E}\left[\left(1 - aH_k + \frac{a^2H_k^2}{2}\right)^2\right] \mathbb{E}[\hat{Y}_{k-1}^2] &=\mathbb{E}\left[1 - 2aH_k + 2a^2H_k^2 - {a^3H_k^3}  + \frac{a^4H_k^4}{4} \right]  \mathbb{E}[\hat{Y}_{k-1}^2]\\
    &= \left( 1 - 2ah + 4a^2h^2 - {6a^3h^3} + {6a^4h^4}\right)^k u_0^2 =: \alpha^k u_0^2.
\end{align*}
Again the variance converges to zero as $k \rightarrow \infty$, iff $|\alpha| < 1$. This  is equivalent to $$ah < \frac13\left(1- \sqrt[3]{\frac{2}{5+\sqrt{29}}} + \sqrt[3]{\frac{5+\sqrt{29}}{2}}\right) \approx 0.7181,$$
which again is a stricter stepsize restriction compared to what we have learnt about the random timestep Euler method.
Now, these results seem surprising at first: given that the second-order stochastic Euler dynamics is broadly motivated by higher-order Runge--Kutta methods, we would have expected that the new stability region is at least not smaller than that of the second-order stochastic Euler dynamics, especially given that the jump chain in the linear setting here is actually equivalent to a random timestep explicit midpoint rule. On the other hand, when approximating the stable trajectory $(u(t))_{t \geq 0}$ by piecewise polynomials, we approximate a stable function by functions that quickly diverge to $\pm \infty$ -- they especially diverge more quickly than the linear functions employed in the random timestep Euler method. If the polynomial pieces diverge more quickly, the interruptions due to jumps in the companion process $(\overline{Y}(t))_{t \geq 0}$ need to happen more frequently. Thus, in this situation, we may obtain a higher convergence order, but the method becomes less stable.

\section{Numerical experiments} \label{Sec_Numerical}
In the following, we illustrate and verify the theoretical results obtained throughout this work. We start with a stable linear ODE in one dimension in Subsection~\ref{Subs_Exp1D} and then move on to an underdamped harmonic oscillator in Subsection~\ref{Subs_ExpOsci}. We also explore two situations not covered by our theory: the local truncation error of second-order stochastic Euler dynamics and the stability of stochastic Euler dynamics in linear ODEs with complex-valued eigenvalues.
\subsection{One-dimensional linear problem} \label{Subs_Exp1D}
We first aim to verify and illustrate the results shown in Theorem~\ref{thm_det}, Theorem~\ref{thm:RMSTE}, and Proposition~\ref{prop_con er} using the one-dimensional linear model problem: $u' = -u, u(0) = 1$ on $X = \mathbb{R}$. In the context of this ODE, we also estimate the local root mean square truncation error of the second-order stochastic Euler dynamics.
\paragraph{Convergence of deterministic Euler dynamics.}
We begin with the deterministic Euler dynamics in the context of the linear model ODE above.
Here, we evaluate the analytical solution of the deterministic Euler dynamics -- the formula is stated in Example~\ref{Ex_Lin_1D_firstorder}. Then, we compute the distances between the deterministic Euler dynamics and the ODE solution $\|w(t) - u(t)\|$ and the distances between the deterministic Euler dynamics and its companion process $\|w(t) - \overline{w}(t)\|$ for $t \in \{0.01, 0.1, 1\}$ and $h \in \{10^{-4}, 1.5\cdot 10^{-4},\ldots,10^{0}\}$. We plot the results in Figure~\ref{fig_DED_error}. The results confirm the linear $O(h; h \downarrow 0)$ rate shown in Theorem~\ref{thm_det}.

\begin{figure}
    \centering
    \includegraphics[scale=0.8]{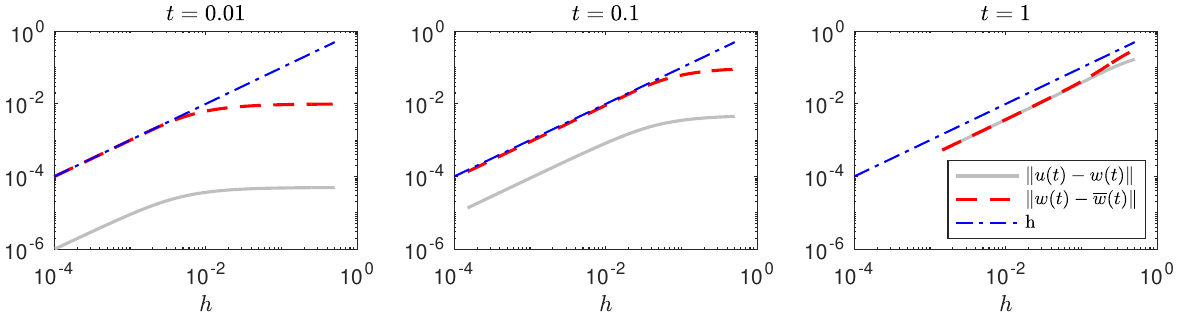}
    \caption{Distances between deterministic Euler dynamics $(w(t))_{t \geq 0}$ and ODE solution $(u(t))_{t \geq 0}$, as well as between deterministic Euler dynamics $(w(t))_{t \geq 0}$ and its companion process $(\overline{w}(t))_{t \geq 0}$ with respect to $h$ at  time points $t \in \{0.01, 0.1, 1\}$. Here, $(u(t))_{t \geq 0}$  and $(w(t),\overline{w}(t))_{t \geq 0}$ correspond to the linear model ODE $u' = -u, u(0) = 1$.}
    \label{fig_DED_error}
\end{figure}

\paragraph{Local root mean square truncation error.}
We simulate the stochastic Euler dynamics for the linear model problem to estimate the local root mean square truncation error $\mathbb{E}[(V(\varepsilon)-u(\varepsilon))^2]^{1/2}$ for $\varepsilon \in \{2^{-8}, 2^{-7},\ldots,2^0\}$ and $h \in \{0.1, 1\}$, as well as $h \equiv \varepsilon$. The expected values are estimated with a Monte Carlo simulation using $10^5$ samples. The results given in Figure~\ref{fig_RMSTE} confirm the second order shown in Theorem~\ref{thm:RMSTE}. In the same figure, we also show estimates of the local root mean square truncation error of the second order stochastic Euler dynamics. Those appear to converge at speed $O(\varepsilon^3; \varepsilon \downarrow 0)$, as we would expect from the theory of ODE solvers.

\begin{figure}
    \centering
    \includegraphics[scale = 0.85]{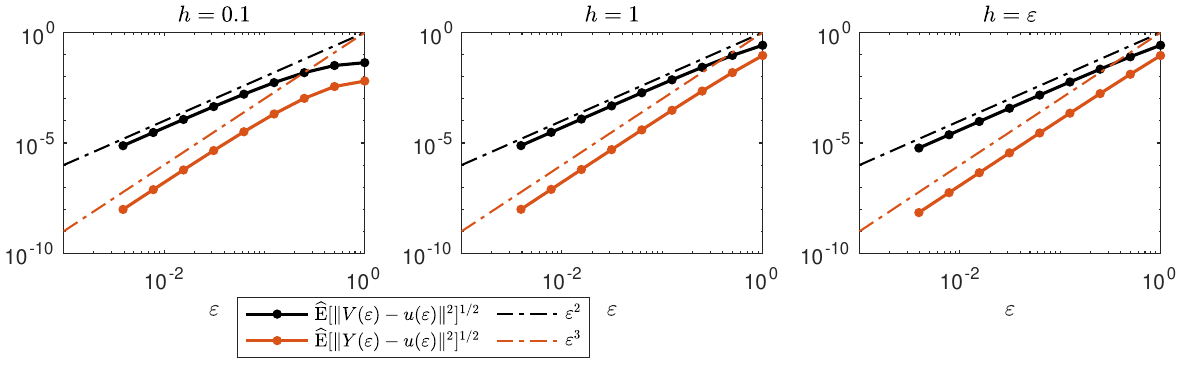}
    \caption{Estimates of the local root mean square truncation error of the stochastic Euler dynamics and second-order stochastic Euler dynamics at multiple time points $\varepsilon > 0$ regarding the linear model problem $u' = -u, u(0) = 1$ for $h \in \{0.1, 1\}$ and $h = \varepsilon$. $\widehat{\mathrm{E}}$ denotes the sample mean and is computed using $10^5$ samples.}
    \label{fig_RMSTE}
\end{figure}

\paragraph{Stability of the stochastic Euler dynamics.}
We simulate the stochastic Euler dynamics for the linear model problem to study their longtime behaviour and verify the result shown in Proposition~\ref{prop_con er}. Indeed, we perform Monte Carlo simulations to estimate $\mathbb{E}[V(t)^2]$ with $t \in \{0,4,\ldots,60\}$ and $h \in \{0.125, 0.25,0.5,1,2\}$ using $10^6$ samples and show the results in Figure~\ref{fig_exponentialrates}. The estimated means in the plot are all sampled independently of each other.
According to Proposition~\ref{prop_con er}, the estimated second moment should stay below the blue $\exp(-2t)$ line for $h \leq 0.25$ and below the green $\exp(-t/(2h))$ line if $h \in [0.25,1)$. Looking at the plots, this appears to be confirmed: there is a slight discrepancy in case $h \in \{0.25, 0.5\}$ for small $t$, which we might be able to explain with the large variance of $V(t)^2$ that leads to a large sampling error. Our theory indicates instability if $h > 1$, which we can see in the case $h = 2$.
\begin{figure}
    \centering
    \includegraphics[scale = 0.82]{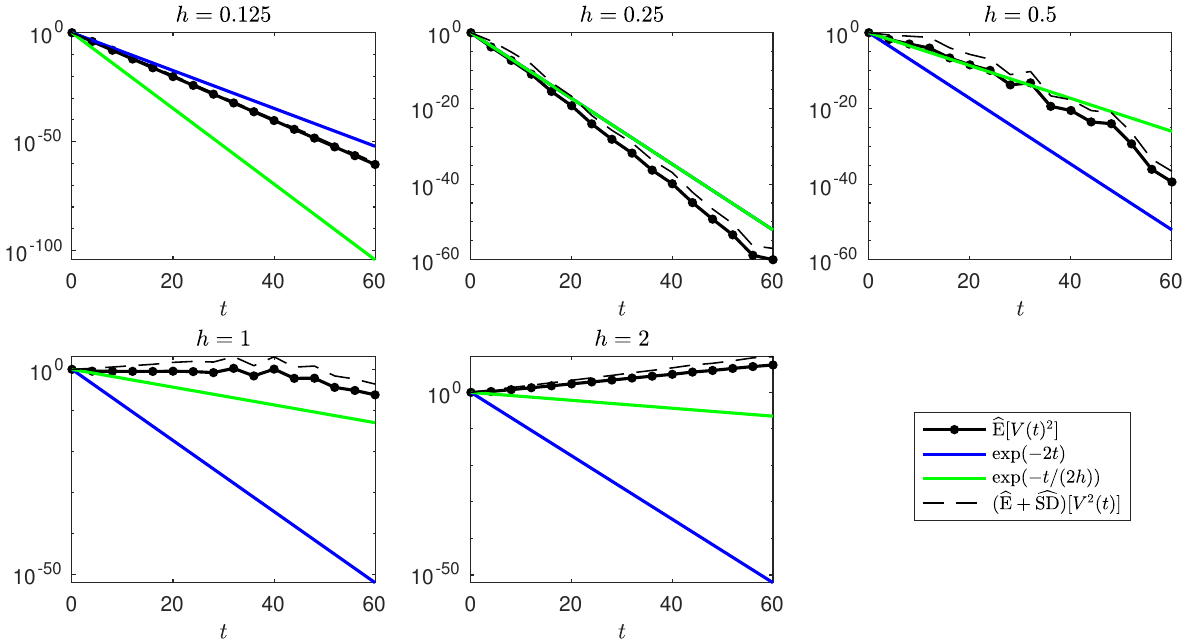}
   \caption{Estimated second moments of the stochastic Euler dynamics of the linear model problem $u' = -u, u(0) = 1$ with large $t$ that show the dynamics' asymptotic behaviour. $\widehat{\mathrm{SD}}$ denotes the sample standard deviation; along with $\widehat{\mathrm{E}}$, it has been computed using $10^6$ samples.}
    \label{fig_exponentialrates}
\end{figure}

\subsection{Underdamped harmonic oscillator} \label{Subs_ExpOsci}
We now consider the underdamped harmonic oscillator $u_1' = u_2, u_2' = -u_1-u_2, u(0) = (1,0)^T$ on $X = \mathbb{R}^2$, see also Figure~\ref{fig:convergence_SED}. We first aim to verify the stability results we have obtained regarding the deterministic Euler dynamics in Theorem~\ref{thm:lin_stab_fo}. We then explore computationally the stability of the stochastic Euler dynamics in this setting; we note that the underdamped harmonic oscillator is not covered by Theorem~\ref{thm_stab_higherdim}.

\paragraph{Stability of the deterministic Euler dynamics.} We simulate the deterministic Euler dynamics regarding the underdamped harmonic oscillator by evaluating the matrix exponential of the matrix $B$ as given in \eqref{eq_B_matrix} at high accuracy. We compute the solution for $t \in \{0, 0.01, \ldots, 40\}$ and $h \in \{0.2, 0.6, 2/3, 0.7\}$ and plot it in Figure \ref{fig_DED_stab}. We see a stable solution for $h < 2/3$ and an unstable solution for $h > 2/3$, as expected from Theorem~\ref{thm:lin_stab_fo}. When choosing $h = 2/3$ we appear to have obtained a dynamical system that is free of friction.
\begin{figure}
    \centering
    \includegraphics[scale = 0.8]{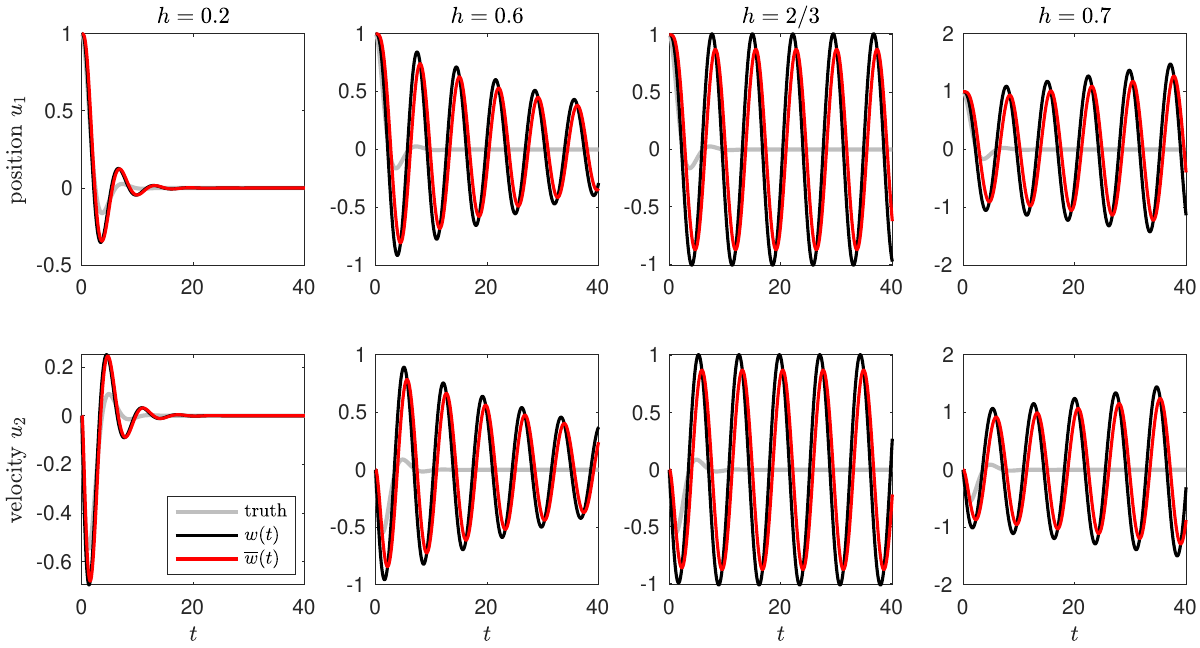}
    \caption{Trajectories of the deterministic Euler dynamics applied to the underdamped harmonic oscillator $u_1' = u_2, u_2' = -u_1-u_2, u(0) = (1,0)^T$ with $h \in \{0.2, 0.6, 2/3, 0.7\}$. }
    \label{fig_DED_stab}
\end{figure}

\paragraph{Stability of the stochastic Euler dynamics.}
We choose the same setup as in the deterministic Euler dynamics above to explore the stability of the stochastic Euler dynamics in the underdamped harmonic oscillator that is not covered by our theory. In Figure~\ref{fig:SED_stab_harmonic}, we plot five independent realisations of the stochastic Euler dynamics regarding the underdamped harmonic oscillator  for each of $h \in \{0.2, 0.6, 2/3, 0.7\}$ on top of each other.
\begin{figure}
    \centering
    \includegraphics[scale = 0.8]{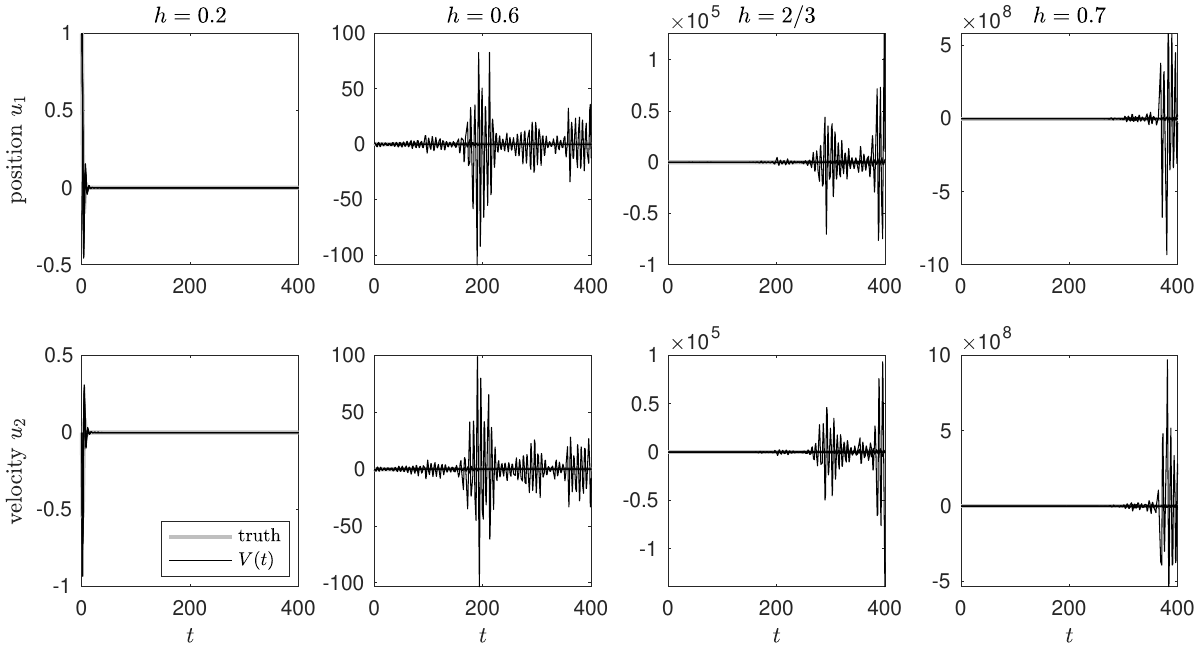}
    \caption{Five realisations of the stochastic Euler dynamics approximating the underdamped harmonic oscillator $u_1' = u_2, u_2' = -u_1-u_2, u(0) = (1,0)^T$ with $h \in \{0.2, 0.6, 2/3, 0.7\}$ .}
    \label{fig:SED_stab_harmonic}
\end{figure}
The shown samples appear to imply a very similar stability behaviour that we have also seen above for the deterministic Euler dynamics; especially stability for small $h$. In Figure~\ref{fig:SED_stab_harmonic_rate}, we study the stability more systematically by a Monte Carlo estimation of $\mathbb{E}[\|V(t)\|^2]$ for $t \in \{0, 50,\ldots,600\}$ with $h \in \{0.2, 0.6, 2/3, 0.7\}$ using $10^5$ samples; again the $\|V(t)\|^2$ are sampled independently for all $t$ and $h$.
There $h \in \{2/3, 0.7\}$ clearly lead to a diverging stochastic Euler dynamics and $h = 0.2$ appears to be very stable. The case $h = 0.6$ that appeared to be eventually stable in  Figure~\ref{fig:SED_stab_harmonic_rate} now appears to be unstable or at least to not match the asymptotic behaviour of the underdamped harmonic oscillator. More work is needed to understand stochastic Euler dynamics in the context of underdamped systems.
\begin{figure}
    \centering
    \includegraphics[scale=0.85]{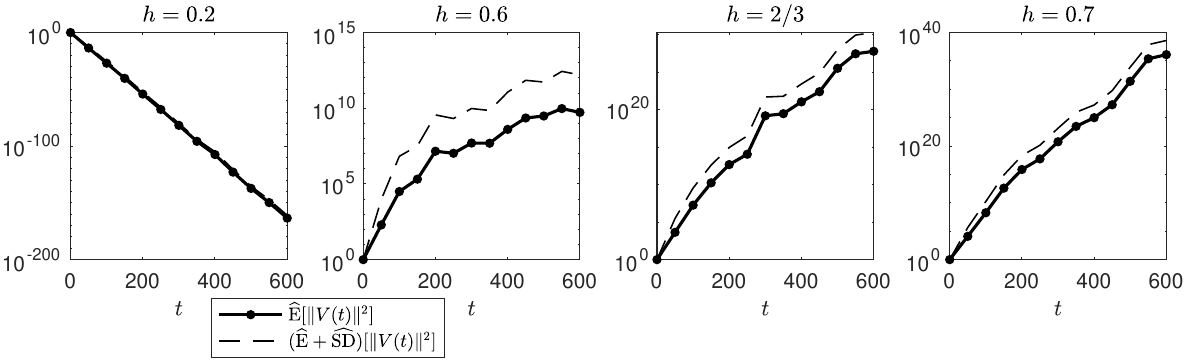}
    \caption{Sample approximations of $\mathbb{E}[\|V(t)\|^2]$ for $t \in \{0, 50,\ldots,600\}$ with $h \in \{0.2, 0.6, 2/3, 0.7\}$ in the context of the underdamped harmonic oscillator $u_1' = u_2, u_2' = -u_1-u_2, u(0) = (1,0)^T$. $\widehat{\mathrm{E}}$ and $\widehat{\mathrm{SD}}$ are computed using $10^5$ samples.}
    \label{fig:SED_stab_harmonic_rate}
\end{figure}

\section{Conclusions and outlook} \label{Sec_Conclusions}
In this work, we have proposed and analysed the stochastic Euler dynamics, a continuous-time Markov process ansatz for the numerical approximation of ordinary differential equations. This ansatz is given as a linear spline interpolation of a random timestep Euler method with i.i.d.\ exponentially distributed timestep sizes. The stochastic Euler dynamics have allowed us to derive novel convergence and stability results regarding random timestep Euler methods: we could prove weak convergence of the stochastic Euler dynamics to the underlying ODE, derive the local root mean square truncation error in the case of linear ODEs, and show stability for certain classes of linear ODEs. In the latter case, we could also compute bounds on the continuous rate of convergence to stationarity of the stochastic Euler dynamics. Numerical experiments helped us to verify these results. We also propose and study deterministic and second-order stochastic Euler dynamics. The former arise from an approximation of the infinitesimal operator of the stochastic Euler dynamics; they recover several properties of the forward Euler method and may be of independent interest.

Our discussion of the second-order stochastic Euler dynamics reveals the difficulty of employing higher order methods -- we have especially noticed stability problems that should be even more pronounced when further increasing the polynomial  order of the splines. Thus, more work is needed in this direction; also and especially when moving to  random timestep Runge--Kutta methods. Moreover, whilst we have focused on rather simple ODEs in this work, random timestep numerical methods are of particular interest in chaotic and irregular dynamical systems. Especially in the context of the randomised discretisation of chaotic dynamical systems, we anticipate the stochastic Euler dynamics to be a helpful tool.

\section*{Acknowledgements}
The author would like to thank the Isaac Newton Institute for Mathematical Sciences for support and hospitality during the programme \emph{The mathematical and statistical foundation of future data-driven engineering} when work on this paper was undertaken (EPSRC grant EP/R014604/1).
\bibliography{library}
\bibliographystyle{plain}
\end{document}